\documentclass{article}

\usepackage{PRIMEarxiv}

\usepackage[utf8]{inputenc} 
\usepackage[T1]{fontenc}    
\usepackage{hyperref}       
\usepackage{url}            
\usepackage{booktabs}       
\usepackage{amsfonts}       
\usepackage{nicefrac}       
\usepackage{microtype}      
\usepackage{lipsum}
\usepackage{fancyhdr}       
\usepackage{graphicx}       
\graphicspath{{media/}}     
\usepackage{amssymb}
\usepackage{amsmath,amsfonts}
\usepackage{mathtools}
\usepackage[capitalise]{cleveref}
\usepackage[toc,page]{appendix}
\usepackage{xcolor}
\usepackage{enumitem}
\usepackage{float}
\usepackage{amsthm}
\usepackage{bm}
 \usepackage{stmaryrd}
\usepackage{graphicx}
\usepackage{verbatim}
\usepackage{subfigure}
\usepackage{multirow}
\usepackage{booktabs}
\newtheorem{theorem}{Theorem}[section]

\newtheorem{hypothesis}{Hypothesis}[section]
\newtheorem{corollary}{Corollary}[section]
\newtheorem{lemma}[theorem]{lemma}

\newtheorem{definition}{Definition}[section]
\newtheorem{proposition}{Proposition}[section]
\title{Deep neural network solutions for oscillatory Fredholm integral equations
\thanks{\textit{\underline{Citation}}: 
\textbf{Authors. Title. Pages.... DOI:000000/11111.}} 
}

\author{
  Jie Jiang \\
  School of Computer Science and Engineering\\
  Sun Yat-sen University \\
  Guangzhou\\
  \texttt{jiangj73@mail2.sysu.edu.cn} \\
   \And
  Yuesheng Xu \\
  Department of Mathematics and Statistics \\
  Old Dominion University \\
  Norfolk\\
  \texttt{y1xu@odu.edu} \\
}

\begin{document}
\maketitle

\begin{abstract}
We studied the use of deep neural networks (DNNs) in the numerical solution of the oscillatory Fredholm integral equation of the second kind. It is known that the solution of the equation exhibits certain oscillatory behaviors due to the oscillation of the kernel. It was pointed out recently that standard DNNs favour low frequency functions, and as a result, they often produce poor approximation for functions containing high frequency components. We addressed this issue in this study. We first developed a numerical method for solving the equation with DNNs as an approximate solution by designing a numerical quadrature that tailors to computing oscillatory integrals involving DNNs. We proved that the error of the DNN approximate solution of the equation is bounded by the training loss and the quadrature error.
We then proposed a multi-grade deep learning (MGDL) model to overcome the spectral bias issue of neural networks. Numerical experiments demonstrate that the MGDL model is effective in extracting multiscale information of the oscillatory solution and overcoming the spectral bias issue from which a standard DNN model suffers. 
\end{abstract}

\keywords{oscillatory Fredholm integral equation \and deep neural network \and spectral bias}

\section{Introduction}
The goal of this paper is to develop an effective deep neural network (DNN) method for the numerical solution of the oscillatory Fredholm integral equation of the second kind. The integral equation has wide applications in physics and engineering, such as  
electromagnetic scattering \cite{2010The, 2003Inverse}. Even though the numerical solution for the equation is a classical research topic  \cite{Atkinson, 2010The, 2015Multiscale, wang2015oscillation, Xiang2023}, it is desirable to investigate how DNNs can be effectively used to represent its numerical solution, given the recent success   of the use of DNNs in the numerical solution of partial differential equations \cite{raissi2018deep, raissi2019physics}  and variational problems \cite{2018The}.

The solution of the oscillatory Fredholm integral equation of the second kind exhibits certain oscillation due to the oscillation of its kernel, see \cite{wang2015oscillation}. In other words, the solution of the equation contains high frequency components as a result of its involvement of an oscillatory kernel. A similar circumstance also exists in the solution of the oscillatory Volterra integral equation \cite{Brunner2015}. This fact places a barrier for using DNNs as an approximate solution of the equation because of the spectral bias phenomenon of neural networks. The spectral bias phenomenon was discovered in \cite{Rahaman2019}, which showed that while neural networks can approximate arbitrary functions, they favour low frequency ones, and as a result, they exhibit a bias towards smooth functions. It is the aim of this paper to investigate how to overcome the spectral bias barrier.

Because of the nonlinearity of DNNs, solving the integral equation for a DNN approximate solution boils down to solve a non-convex optimization problem by an iterative scheme. Every step of the iteration requires to evaluate the objective function which involves oscillatory integrals. Hence, it is inevitable to compute oscillatory integrals at every step of the iteration. There is a large literature on the numerical quadrature of oscillatory integrals. Computing integrals of the product of a smooth non-oscillatory function and a structured oscillatory kernel function has a long history in numerical analysis. For example, computing integrals involving the Fourier kernel, the Bessel kernel and the Gauss kernel was studied in \cite{Arieh2005Efficient, 1960A},  \cite{2004Stability, 2011Asymptotic} and \cite{MaXu2018}, respectively.
Existing quadrature methods for oscillatory integrals may be divided into four major categories: asymptotic methods \cite{Arieh2005Efficient}, Filon-type methods \cite{1960A, 2015Computing, MaXu2018}, Levin-type methods \cite{1997Analysis, Olver2009GMRES, 2010Fast} and numerical steepest descent methods \cite{1984Asymptotic}. All these methods require to know the oscillation level of the integrand in advance. 
The oscillatory integrals appearing in the context of using DNNs as approximate solutions are generated during the iteration and thus, their oscillatory levels cannot be determined before hand. Therefore, aforementioned methods cannot be applied directly to the current situation. We will design a numerical quadrature that tailors to oscillatory integrals generated during the iteration process and estimate its error bound. Moreover, we will establish the error estimate that bounds the error of the DNN solution by the training loss and the quadrature error.

It is well understood that DNNs have the great expressiveness to present different scales of a function. However, the standard DNN model is not always competent to extract multiscale information of a function as discussed previously. The multi-grade deep learning (MGDL) model tailors to the needs of the extraction. The effectiveness of the MGDL model was demonstrated in \cite{Xu:2023aa} and \cite{XuZeng2023} in the context of function approximation and numerical solutions of partial differential equations, respectively. We will develop a MGDL model for the DNN solution of the oscillatory Fredholm integral equation and demonstrate numerically that the proposed MGDL model can effectively extract multiscale components of its oscillatory solution, leading to a promising method, unlike the standard DNN model which suffers from the spectral bias.

We organize this paper in eight sections. In section 2,  we outline the oscillatory Fredholm integral equations under consideration and discuss the oscillatory property of its solution. In  section 3, we describe the DNN model for the numerical solution of the oscillatory integral equation and its associated optimization problem.  Section 4 is devoted to the development of a numerical quadrature scheme for computing the oscillatory integrals involving DNNs generated during the iteration that solves the optimization problem, and its error analysis.  In section 5, we establish the error estimate for the DNN approximate solution bounded by the training loss and the quadrature error. We describe in section 6 the MGDL model for the numerical solution of the integral equation. Numerical results are presented in section 7. Finally, we make conclusive remarks in section 8.

\section{Oscillatory Fredholm Integral Equation of the Second Kind}
In this section, we describe the Fredholm integral equation of the second kind with an oscillatory kernel to be considered in this paper. 

Let $I:=[-1,1]$. We denote by $C(I)$ the space of continuous complex-valued functions defined on $I$ and $C(I^2)$ the space of continuous complex-valued bivariate functions on $I^2$. Suppose that $K\in C(I^2)$ and $f\in C(I)$ are given. We consider the oscillatory integral equation 
\begin{equation}\label{fredholm_equation}
y(s)-\int_{I} K(s,t)e^{i\kappa |s-t|}y(t)\mathrm{d}t=f(s), \ \ s\in I, 
\end{equation}
where  $\kappa\geqslant 1$ is the wavenumber and $y\in C(I)$ is the solution to be solved. 
For simplicity, in this paper we always assume the kernel function
\begin{equation*}
    K(s,t):=\lambda,\ \ (s,t)\in I\times I
\end{equation*}
where $\lambda\in \mathbb{C}$. By defining the integral operator $\mathcal{K}$ for $h\in C(I)$ by
\begin{equation*}
(\mathcal{K}h)(s):=\int_{I} e^{i\kappa |s-t|}h(t)\mathrm{d}t, \ \  s\in I, \label{def_K}
\end{equation*} 
integral equation \eqref{fredholm_equation} can be written in its operator form
\begin{equation}
    (\mathcal{I}-\lambda\mathcal{K})y=f, \label{fredholm_equation_operator}
\end{equation}
where $\mathcal{I}$ denotes the identity operator on $C(I)$.

It is known that the solution $y$ of equation \eqref{fredholm_equation_operator} in general exhibits certain oscillatory behaviour. It is advantageous to classify the degree of oscillation for a given function.
We recall a definition from   \cite{wang2015oscillation}. 

\begin{definition}
     Let $n$ be a positive number and $(X, \|\cdot\|_X)$ be a normed space. A function $u$ is called $\kappa$-oscillatory of order $n$ in $X$ if it satisfies
 \begin{enumerate}
     \item  $u$ is $\kappa$-oscillatory in $X$,
     that is, $u$ is $\kappa$-parameterized and $u\in X$ for any $\kappa$,
     \item there exist a positive constant $c$  such that for all $\kappa>1$,
     $$ \kappa^{-n}\|u\|_X\leqslant c.$$
\end{enumerate}
When $n=0$, we say that $u$ is non-$\kappa$-oscillatory in $X$.
\end{definition}
 Let $m$ be a fixed positive integer. For Sobolev space $H^m(I):=\{u\in L^2(I):u^{(n)}\in L^2(I), n\in \mathbb{Z}_{m+1}\}$ where $\mathbb{Z}_{m+1}:=\{0,1,2,\dots,m\}$ with the norm
$$ \|u\|_{H^m}:=\left(\sum_{j\in \mathbb{Z}_{m+1}}\left\|u^{(j)}\right\|_2^2\right)^{1/2},$$
define the space of the oscillatory functions $H^m_{\kappa,0}(I)$ as
\begin{equation*}
    H^m_{\kappa,0}(I):=\{u_1+u_2e^{i\kappa \cdot}+u_3e^{-i\kappa \cdot}: u_j\ \ \mbox{is non-$\kappa$-oscillatory in}\ \ H^m(I),\ j\in \mathbb{N}_3 \}.
\end{equation*}
where $\mathbb{N}_n:=\{1,2,\dots,n\}$. Then Theorem 5.8 of \cite{wang2015oscillation} shows that the solution $y$ exhibits certain oscillatory behaviour. We state the result below for convenience of reference.

\begin{theorem}\label{Solution:oscilation}
Let $y$ be the solution of equation \eqref{fredholm_equation}. If $K\in C^{[m]}(I^2)$ where $C^{[n]}(I^2):=\{L\in C^n(I^2):max_{p,q\in \mathbb{Z}_{n+1}}\|L^{(p,q)}\|<\infty\}$ is independent of $\kappa$ and $f\in H^m_{\kappa,0}(I)$, then $y\in H^m_{\kappa,0}(I)$.
\end{theorem}

This theorem shows the solution $y$ is in the space of oscillatory function $H^m_{\kappa,0}(I)$, which is the base for our further discussion and numerical experiments.

For numerical solutions of the integral equation \eqref{fredholm_equation_operator}, traditional methods typically focus on constructing an appropriate function space where a viable solution exists, and subsequently seek the solution within this space. Commonly utilized function spaces encompass polynomial spaces, trigonometric function spaces, piecewise polynomial spaces, and multi-scale spaces. The construction of a function space typically necessitates a profound understanding of the solution's characteristics.
In contrast, the DNN method directly seeks the solution through training data. Specifically, this methodology employs a DNN with unknown parameters to approximate the desired solution. The equation and constraints are then transformed into an optimization problem. Finally, the optimal parameters of the network are determined. In the subsequent section, we will introduce the DNN method to solve the equation \eqref{fredholm_equation_operator}.

\section{DNN Learning Model}
In this section, we describe a DNN learning model for solving the oscillatory Fredholm integral equation. 

We begin by describing DNNs to be used as approximate solutions of the integral equation.  We adopt the notation from \cite{Xu:2021aa, XuZhang2023}.
A DNN is  a function formed by compositions of vector-valued functions, each of which is defined by an activation function applied to an affine map.
Given a uni-variate activation function $\sigma:\mathbb{R}\to \mathbb{R}$, the vector-valued activation function $\boldsymbol{\sigma}:\mathbb{R}^d\to \mathbb{R}^d$ is defined as
$$
\boldsymbol{\sigma}(\mathbf{v}):= [\sigma(v_1), \sigma(v_2), \ldots, \sigma(v_d)]^T, \ \ \mbox{for}\ \ \mathbf{v} := [v_1, v_2, \ldots, v_{d}]^T\in \mathbb{R}^{d}.
$$
For vector-valued functions $f_j$, $j\in \mathbb{N}_n$, satisfying the condition that the range of $f_j$ is contained in the domain of $f_{j+1}$, their  consecutive composition is denoted by
$$
\bigodot_{j=1}^n f_j:=f_n\circ f_{n-1}\circ \dots \circ f_1.
$$
Suppose that positive integers $m_j$, $j\in \mathbb{Z}_{n+1}$,  with $m_0:=1, m_n:=2$, are chosen.
For weight matrices $\mathbf{W}_j\in \mathbb{R}^{m_j\times m_{j-1}}$ and the bias vectors $\mathbf{b}_j\in \mathbb{R}^{m_j}$, $j\in \mathbb{N}_{n}$, a DNN is a function defined by
\begin{equation}
    \mathcal{N}_n(\{\mathbf{W}_j, \mathbf{b}_j\}_{j=1}^n;s) := \left(\mathbf{W}_n\bigodot_{j=1}^{n-1}\boldsymbol{\sigma}(\mathbf{W}_j\cdot+\mathbf{b}_j)+\mathbf{b}_n\right)(s), \quad s\in I. \label{dnn_output}
\end{equation}
In particular, the output of the last hidden layer is called the feature of the DNN, which is defined by
\begin{equation}
    \mathcal{F}_{n-1}(\{\mathbf{W}_j, \mathbf{b}_j\}_{j=1}^{n-1};s) := \left(\bigodot_{j=1}^{n-1}\boldsymbol{\sigma}(\mathbf{W}_j\cdot+\mathbf{b}_j)\right)(s), \quad s\in I.  \label{dnn_feature}
\end{equation}

The goal of this paper is to construct approximate solutions, of integral equation \eqref{fredholm_equation}, which take the form of \eqref{dnn_output}. Note that the solution of equation \eqref{fredholm_equation} is a complex-valued function. This requires us to define an operator $\mathcal{T}$ that transforms a real two dimensional vector-valued function to a complex-valued function.
Specifically, for an $\mathbf{f}(s):=[f_1(s), f_2(s)]^T$, $s\in I$, with $f_1, f_2$ being real-valued functions defined on $I$, the operator $\mathcal{T}$ is defined as 
$(\mathcal{T} \mathbf{f})(s):=f_1(s)+if_2(s)$, $s\in I$.
With this operator, we define the loss function as
\begin{flalign}
    e\left(\{\mathbf{W}_j, \mathbf{b}_j\}_{j=1}^{n};s\right):= 
    (\mathcal{I}-\lambda\mathcal{K}) \mathcal{T}\mathcal{N}_n(\{\mathbf{W}_j, \mathbf{b}_j\}_{j=1}^{n};\cdot) (s)-f(s),\quad s\in I. \label{hat_single_grade_error_function_K}
\end{flalign}
We then find the optimal parameters $\{\mathbf{W}^*_j, \mathbf{b}^*_j\}_{j=1}^n$ by solving the  optimization problem
\begin{flalign}
\min_{\{\mathbf{W}_j, \mathbf{b}_j\}_{j=1}^n} \left\|e\left(\{\mathbf{W}_j, \mathbf{b}_j\}_{j=1}^{n};\cdot\right)\right\|_2^2.  \label{hat_DNN_opt_K}
\end{flalign}
Once the optimal parameters $\{\mathbf{W}^*_j, \mathbf{b}^*_j\}_{j=1}^n$ is obtained, the function
$y^*:=\mathcal{T}\mathcal{N}_{n}(\{\mathbf{W}_j^*, \mathbf{b}_j^*\}_{j=1}^n;\cdot)$
provides a DNN solution of equation \eqref{fredholm_equation}.

Numerical implementation for solving optimization problem \eqref{hat_DNN_opt_K} requires the availability of collocation data $\{(x_j,f(x_j))$, $j\in \mathbb{N}_N\}$ and discretization of the integral operator \eqref{def_K}. 
The $L_2$ norm used in equation \eqref{hat_DNN_opt_K} may be approximated by its discrete form
\begin{equation*}
    \|\phi\|_N:=\sqrt{\frac{1}{N}\sum_{j\in \mathbb{N}_N} |\phi(x_j)|^2}, \ \ \mbox{for}\ \ 
    \phi \in C(I),
\end{equation*}
where $| \cdot |$ is the modulus of a complex number. 
It is clear that the functional $\|\cdot\|_N$ is a semi-norm on $C(I)$, but it is not a norm since $\|\phi\|_N=0$ does not imply $\phi=0$. This semi-norm $\|\cdot\|_N$ is associated with the $l_2$ norm in $\mathbb{C}^N$. For any $\phi\in C(I)$, defining $\mathbf{v}_\phi:=[\phi(x_j), j\in \mathbb{N}_N]$, it follows from the definition of $\|\cdot\|_N$ that
\begin{equation}
    \|\phi\|_N=\frac{\|\mathbf{v}_\phi\|_{\ell_2}}{\sqrt{N}}. \label{relation_N_2}
\end{equation} 
We then assume that there is a discrete oscillatory integral operator $\mathcal{K}_{p_{_\kappa}}$ that approximates the operator $\mathcal{K}$ defined by \eqref{def_K}, where $p_\kappa$ is a positive integer dependent on $\kappa$. We postpone the construction of operator $\mathcal{K}_{p_{_\kappa}}$ to the next section. 

With the availability of the discrete form of the $L_2$-norm and the discrete oscillatory integral operator $\mathcal{K}_{p_{_\kappa}}$, the discrete loss function is defined by
 \begin{flalign}
    \tilde{e}\left(\{\mathbf{W}_j, \mathbf{b}_j\}_{j=1}^{n};s\right):= \left(f-
    (\mathcal{I}-\lambda\mathcal{K}_{p_{_\kappa}}) \mathcal{T}\mathcal{N}_n(\{\mathbf{W}_j, \mathbf{b}_j\}_{j=1}^{n};\cdot)\right) (s),\quad  s\in I, \label{single_grade_error_function}
\end{flalign}
which approximates the loss function defined by \eqref{hat_single_grade_error_function_K}.
We then find the optimal parameters $\{\tilde{\mathbf{W}}^*_j, \tilde{\mathbf{b}}^*_j\}_{j=1}^n$ by solving the discrete optimization problem
\begin{flalign}
\arg\min_{\{\mathbf{W}_j, \mathbf{b}_j\}_{j=1}^n} \left\|\tilde{e}\left(\{\mathbf{W}_j, \mathbf{b}_j\}_{j=1}^{n};\cdot\right)\right\|_{N}^2,  \label{DNN_opt}
\end{flalign}
where
\begin{equation*}
    \left\|\tilde{e}\left(\{\mathbf{W}_j, \mathbf{b}_j\}_{j=1}^{n};\cdot\right)\right\|_{N}^2:=\frac{1}{N}\sum_{l\in\mathbb{N}_N}|\tilde{e}\left(\{\mathbf{W}_j, \mathbf{b}_j\}_{j=1}^{n};x_l\right)|^2.
\end{equation*}
Upon finding the parameters $\{\tilde{\mathbf{W}}^*_j, \tilde{\mathbf{b}}^*_j\}_{j=1}^n$, we obtain a numerical solution of equation \eqref{fredholm_equation_operator} given by 
\begin{equation}
     \tilde{y}^*:=\mathcal{T}\mathcal{N}_{n}(\{\tilde{\mathbf{W}}_j^*, \tilde{\mathbf{b}}_j^*\}_{j=1}^n;\cdot) \label{def_fin_y_star}
\end{equation}
with an error defined by
\begin{equation}
\tilde{e}^*(s):= \tilde{e}\left(\{\tilde{\mathbf{W}}_j^*, \tilde{\mathbf{b}}_j^*\}_{j=1}^{n};s\right),\quad s\in I. \label{def_fin_e}
\end{equation}

Optimization problem \eqref{DNN_opt} is often solved by employing gradient-based algorithms. This motivates us to construct the discrete operator $\mathcal{K}_{p_{_\kappa}}$ via a numerical integration method.

\section{Numerical Quadrature of Oscillatory Integrals}
This section is devoted to the development of the discrete oscillatory integral operator $\mathcal{K}_{p_{_\kappa}}$ and its error estimates.

The discrete oscillatory integral operator $\mathcal{K}_{p_{_\kappa}}$ that approximates the  oscillatory integral operator $\mathcal{K}$ is a key attribute of the DNN model. The discrete operator $\mathcal{K}_{p_{_\kappa}}$ should effectively approximate the continuous oscillatory integral operator $\mathcal{K}$. Typically, in the optimization process, the optimization problem \eqref{DNN_opt} is solved by using gradient-based optimization algorithms. These algorithms iteratively update the model parameters $\theta_j:=\{\mathbf{W}_{l,j}, \mathbf{b}_{l,j}\}_{l=1}^{n}$, where $j\in \mathbb{N}:=\{1,2,3,\dots\}$. By defining 
\begin{equation*}
    g_{j}(s):= \mathcal{T} \mathcal{N}_n(\theta_j ;s)\in C(I), \quad j\in \mathbb{N}, \quad s\in I, 
\end{equation*}
we will construct $\mathcal{K}_{p_{_\kappa}}$ in a way that the errors $| (\mathcal{K} g_{j}-\mathcal{K}_{p_{_\kappa}} g_{j}) (x_l)|, l\in \mathbb{N}_N, j\in \mathbb{N}$ are as small as possible. Noting that for any $l\in \mathbb{N}_N$, 
\begin{equation*}
    |(\mathcal{K} g_{j}-\mathcal{K}_{p_{_\kappa}} g_{j}) (x_l)|\leqslant \|\mathcal{K} g_{j}-\mathcal{K}_{p_{_\kappa}} g_{j}\|_\infty, \quad j\in \mathbb{N}, 
\end{equation*} 
our objective is to construct the operator $\mathcal{K}_{p_{_\kappa}}$ to control $\|\mathcal{K}\chi-\mathcal{K}_{p_{_\kappa}} \chi\|_\infty$
when wavenumber $\kappa$ is large for $\chi$ in the class of functions $\{g_{j}\}_{j\in \mathbb{N}}$.

The behavior regarding the level of oscillation in the function sequence $\{g_{j}\}_{j\in \mathbb{N}}$ is a crucial consideration, especially given the use of the $\sin$ function as the activation function in the DNN, which directly influences the oscillatory behavior through the parameter $\theta_j, j\in \mathbb{N}$.  The initial guess $\theta_1$ proposed in \cite{he2015delving} leads to a non-oscillatory function $g_1$ independent of $\kappa$, setting the starting point for the optimization.  In the subsequent optimization process, if the model runs normally, the oscillation of the generated function sequence $\{g_{j}\}_{j\in \mathbb{N}}$ will not deviate significantly from the oscillation of the exact solution $y$. Below, we define a class of oscillatory functions with a relaxation oscillation level so that the class contains the exact solution $y$ and approximate solutions $g_j$ of the $j$-th iteration steps.

Motivated from Theorem \ref{Solution:oscilation}, we assume that there exist functions $u_j$ , $j\in \mathbb{N}_3$, 
satisfying the condition for some  $\tau>0$,
\begin{equation*}
|u_j^{(l)}(s)| \leqslant \tau , \quad \mbox{for all}\ \ s\in I, \  \ l\in \mathbb{Z}_{m+1}, 
\end{equation*}
such that the solution $y$ of integral equation \eqref{fredholm_equation_operator} is in $H^m_{\kappa, 0}(I)$. In other words, $y$ can be expressed as
\begin{equation}
    y(s) = u_1(s) + u_2(s) e^{i\kappa s} + u_3(s) e^{-i\kappa s}, \quad s\in I. \label{solution}
\end{equation}
The functions $g_j$ are not exactly in $H^m_{\kappa, 0}(I)$ but in its perturbation. Below, we define a perturbation class of $H^m_{\kappa, 0}(I)$ that covers the functions $g_j$. 
Let $\Gamma \geqslant 0$ and $r\in \mathbb{N}$. Suppose that  $\alpha_j\in \mathbb{R}$, $j\in \mathbb{N}_{r}$, are unknown but satisfy $|\alpha_j|\leqslant 1+\Gamma$, and functions $w_j$, $j\in \mathbb{N}_{r}$, satisfy
 \begin{equation}
     |w_j^{(l)}(s)|\leqslant \tau, \quad s\in I, \ \ l\in \mathbb{Z}_{m+1}. \label{requirement}
\end{equation}
The functions $\chi_{\kappa}$ in the perturbation class have the form 
\begin{equation}
    \chi_{\kappa}(s):=\sum_{j=1}^r w_j(s) e^{i\alpha_j\kappa s}, \quad s\in I.\label{def_chi}
\end{equation}

We next introduce a sequence of discrete operators $\mathcal{K}_{p}$, for $p \in \mathbb{N}$, which approximate the integral operator $\mathcal{K}$ by using the compound trapezoidal quadrature formula.
Specifically, for each $p \in \mathbb{N}$, let $h:=2/p$, $s_j:=-1+jh$ for $j\in\mathbb{Z}_{p+1}$, and  for any $F\in C(I)$, we define
\begin{equation}
    (\mathcal{K}_{p} F)(s):= \frac{h}{2} \left[F(s_0)e^{i\kappa|s-s_0|}+2\sum_{j=1}^{p-1} F(s_j)e^{i\kappa|s-s_j|}+F(s_p)e^{i\kappa|s-s_p|}\right], \quad s\in I. \label{def_mathcal_K_p}
\end{equation}
In the rest of this section, we will choose the value of $p$ according to the growth of the error $\|\mathcal{K} \chi_{\kappa}-\mathcal{K}_{p}\chi_{ \kappa}\|_\infty$ with respect to the wavenumber $\kappa$. 
Error analysis of the compound trapezoidal quadrature formula for smooth non-oscillatory integrands is well-understood (for example, see \cite{davis2007methods}). However, since the integrand that we consider here is oscillatory and non-differentiable, we will estimate the error by taking the special property of the integrand into account. 

When analyzing the error $\|\mathcal{K} \chi_{\kappa}-\mathcal{K}_{p}\chi_{ \kappa}\|_\infty$, due to \eqref{def_chi}, 
it suffices to consider 
\begin{equation}
    \tilde{\chi}_\kappa(t):=w(t)e^{i\alpha \kappa t}, \quad t\in I, \label{def_tilde_chi}
\end{equation} 
where $\alpha\in \mathbb{R}$ satisfies $|\alpha|\leqslant 1+\Gamma$, and  function $w$ is $m$ times differentiable and satisfies \eqref{requirement}. Our goal is to develop an effective  quadrature scheme for numerical computation of the integral
\begin{equation*}
(\mathcal{K} w(\cdot)e^{i\alpha \kappa \cdot})(s):=\int_{I} \left[w(t) e^{i\kappa\alpha t}\right]e^{i\kappa |s-t|} \mathrm{d}t, \ \  s\in I, 
\end{equation*} 
for the purpose of constructing the discrete operator $\mathcal{K}_{p_\kappa}$.

Next, we estimate $\|\mathcal{K}\tilde{\chi}_{\kappa}-\mathcal{K}_{p} \tilde{\chi}_{\kappa}\|_\infty$.
To this end, for a fixed $s\in I$, we define 
\begin{equation}
    \phi (t):=w(t)e^{i \kappa (\alpha t+|s-t|)} , \quad t\in I.  \label{def_phi}
\end{equation}
Then,  $(\mathcal{K}\tilde{\chi}_{\kappa}-\mathcal{K}_{p} \tilde{\chi}_{\kappa})(s)$ can be expressed as
\begin{equation}
    (\mathcal{K}\tilde{\chi}_{\kappa}-\mathcal{K}_{p} \tilde{\chi}_{\kappa})(s)= \sum_{j=0}^{p-1}\left[\int_{s_j}^{s_{j+1}} \phi(t)\mathrm{d}t-\frac{h}{2} \left(\phi(s_j)+\phi(s_{j+1})\right)\right], \label{prop1_equ1}
\end{equation}
where $h=\frac{2}{p}$, $s_j= -1+jh$ for $j\in \mathbb{Z}_{p+1}$. Thus for any $j\in \mathbb{Z}_p$, it yields that
$s_{j+1}-s_j = h$.

Since the integrand $\phi$ defined by \eqref{def_phi} is $m$ times differentiable at $I\setminus\{s\}$ for any positive integer $m$, we partition interval $I$ into three subintervals such that except for the interval containing the point $s$, the function $\phi$ remains $m$ times differentiable over the other two intervals.
Specifically, we choose 
$$
j^*:= \begin{cases}
 \frac{s+1}{h}-1, &\textrm{if\,} \frac{s+1}{h} \in \mathbb{N},\\
 \lfloor \frac{s+1}{h}\rfloor, &\textrm{otherwise,}
\end{cases}
$$       
where for $x\in \mathbb{R}$, $\lfloor x\rfloor$ represents the largest integer not exceeding $x$. Thus, $j^*\in \mathbb{Z}_{p}$. Noticing
\begin{equation}
|s-t|=\begin{cases}
s-t, & t\in [s_0, s_{j^*}],\\
t-s,    & t\in [s_{j^*+1}, s_{p}],
\end{cases} \label{val_s-t}
\end{equation}
together with the definition \eqref{def_phi} of function $\phi$, we conclude that $\phi$ is analytic in the intervals $[s_0, s_{j^*}]$ and $[s_{j^*+1}, s_{p}]$. 
If we define an $m$-times differentiable function 
\begin{equation}
    \psi(t):=\begin{cases}
        w(t)e^{i(\alpha-1)\kappa t}, & t\in [s_0, s_{j^*}],\\
        w(t)e^{i(\alpha+1)\kappa t}, & t\in [s_{j^*+1}, s_{p}],
    \end{cases}\label{def_psi}
\end{equation}
together with the relation \eqref{val_s-t}, we obtain
\begin{equation*}
    \phi(t)=\begin{cases}
        e^{i\kappa s}\psi(t), & t\in [s_0, s_{j^*}],\\
        e^{-i\kappa s}\psi(t), & t\in [s_{j^*+1}, s_{p}].
    \end{cases}
\end{equation*}
Upon defining
\begin{flalign}
    T_1:=e^{i\kappa s}\sum_{j=0}^{j^*-1}\left[\int_{s_j}^{s_{j+1}} \psi(t)\,\mathrm{d}t-\frac{h}{2} \left(\psi(s_j)+\psi(s_{j+1})\right)\right], \label{def_T1}
\end{flalign}
\begin{equation}
    T_2:=\int_{s_{j^*}}^{s_{j^*+1}} \phi(t)\mathrm{d}t-\frac{h}{2} \left(\phi(s_{j^*})+\phi(s_{j^*+1})\right), \label{def_T2}
\end{equation}
\begin{flalign}
    T_3:=e^{-i\kappa s}\sum_{j=j^*+1}^{p-1}\left[\int_{s_j}^{s_{j+1}} \psi(t)\,\mathrm{d}t-\frac{h}{2} \left(\psi(s_j)+\psi(s_{j+1})\right)\right], \label{def_T3}
\end{flalign}
the equation \eqref{prop1_equ1} can be rewritten as
\begin{equation*}
    (\mathcal{K}\tilde{\chi}_{\kappa}-\mathcal{K}_{p} \tilde{\chi}_{\kappa})(s)= T_1+T_2+T_3. 
\end{equation*}
Using the triangle inequality in the above equation, it yields that 
\begin{flalign}
|(\mathcal{K}\tilde{\chi}_{\kappa}-\mathcal{K}_{p} \tilde{\chi}_{\kappa})(s)|\leqslant  |T_1|+|T_2|+|T_3|. \label{estimate_T1_T2_T3}
\end{flalign}
We will estimate $T_j$, $j=1,2,3$, separately.

We first estimate $|T_2|$ in the following lemma.

\begin{lemma}\label{lem_T2}
If $T_2$ is defined in equation \eqref{def_T2}, then
\begin{equation*}
    |T_2| \leqslant \frac{4\tau}{p} .
\end{equation*}
\end{lemma}

\begin{proof}
Applying the triangle inequality to the definition of $T_2$, it yields that
\begin{flalign*}
    |T_2| 
    \leqslant  \int_{s_{j^*}}^{s_{j^*+1}}  \left| \phi(t)\right|\,\mathrm{d}t+ h\|\phi\|_\infty
    \leqslant (s_{j^*+1}-s_{j^*}+h)\|\phi\|_\infty.
\end{flalign*}
This together with the relation $s_{j^*+1}-s_{j^*}=h$ leads to the bound
\begin{equation}
    |T_2| \leqslant 2h\|\phi\|_\infty. \label{prop1_equ3}
\end{equation}
Using the inequality \eqref{requirement} with $l:=0$, we know that $\|w\|_\infty \leqslant \tau$. Combining this with the definition of $\phi$, we observe that  $\|\phi\|_\infty=\|w\|_\infty\leqslant \tau$. Thus, the inequality \eqref{prop1_equ3} reduces to
$|T_2| \leqslant 2\tau h$, which
together with the fact $h=2/p$ yields the desired result.
\end{proof}

We next estimate the terms $|T_1|$ and $|T_3|$. In the following consideration, we unify the cases $|T_1|$ and $|T_3|$ into one.
For $\mu, \nu\in \mathbb{Z}_{p}$ satisfying $0\leqslant \mu\leqslant \nu\leqslant p-1$ with $j^*\notin [\mu,\nu]$, we consider
\begin{equation}
T:=\sum_{j=\mu}^{\nu}\left[\int_{s_j}^{s_{j+1}} \psi(t)\,\mathrm{d}t-\frac{h}{2} \left(\psi(s_j)+\psi(s_{j+1})\right)\right]. \label{def_T}
\end{equation}
Clearly, if we choose $\mu:=0, \nu:=j^*-1$, then $|T|=|T_1|$ by equation \eqref{def_T1}, and if we choose $\mu:=j^*+1, \nu:=p-1$, then $|T|=|T_3|$ by equation \eqref{def_T3}.

To estimate $|T|$, we will make use of the Taylor expansion of the function $\psi$ at the midpoint of each subinterval into the definition of $T$. Specifically, for each $j\in \mathbb{Z}_{[\mu,\nu]}:=\{\mu, \mu+1, \dots, \nu\}$, we let $s_{j+1/2}:=\frac{1}{2}(s_j+s_{j+1})$, and write $\psi^{(l)}_{j+1/2}:=\psi^{(l)}(s_{j+1/2})$ for $l\in \mathbb{Z}_{m+1}$. By the Taylor theorem, for each $t\in [s_j, s_{j+1}]$, there exists a $\xi_t$ between $t$ and $s_{j+1/2}$ such that
\begin{equation}
    \psi(t) =  \psi_{j+1/2}^{(0)} + \sum_{l=1}^{m-1} \frac{(t-s_{j+1/2})^l}{l!}\psi^{(l)}_{j+1/2} + \frac{(t-s_{j+1/2})^{m}}{m!} \psi^{(m)}(\xi_t). \label{Taylor}
\end{equation}
For each $j\in \mathbb{Z}_{[\mu,\nu]}$, we let
$$
D_{j,1} := \frac{(-h)^m}{2^{m}m!} \psi^{(m)}(\xi_{s_j}) + \frac{h^m}{2^m m!} \psi^{(m)}(\xi_{s_{j+1}}).
$$
By the Taylor expansion \eqref{Taylor},  we obtain that
\begin{equation}
    \frac{h}{2}[\psi(s_j) + \psi(s_{j+1})] = h\psi_{j+1/2}^{(0)} + \sum_{l=1}^{\tilde{m}} \frac{2}{(2l)!}\left(\frac{h}{2}\right)^{2l+1} \psi_{j+1/2}^{(2l)} + \frac{hD_{j,1}}{2},
    \label{lem_estimation1_T_equ1_else}
\end{equation}
where $\tilde{m}:=\lfloor (m-1)/2\rfloor$.
Meanwhile, 
by integrating both sides of equation \eqref{Taylor} from $s_j$ to $s_{j+1}$, and letting
$$
D_{j,2} := \int_{s_j}^{s_{j+1}} \frac{(t-s_{j+1/2})^m}{m!} \psi^{(m)}(\xi_t)\mathrm{d}t,
$$
we have that
\begin{equation}
    \int_{s_j}^{s_{j+1}} \psi(t)\,\mathrm{d}t = h\psi_{j+1/2}^{(0)} + \sum_{l=1}^{\tilde{m}} \frac{2}{(2l+1)!}\left(\frac{h}{2}\right)^{2l+1} \psi_{j+1/2}^{(2l)} + D_{j,2}.\label{lem_estimation1_T_equ2_else}
\end{equation}
Substituting equations \eqref{lem_estimation1_T_equ1_else} and \eqref{lem_estimation1_T_equ2_else} into the right-hand side of equation \eqref{def_T} yields
\begin{equation}
    T = \sum_{j=\mu}^{\nu}\left(-\sum_{l=1}^{\tilde{m}}\frac{2l h^{2l+1}}{4^l(2l+1)!}\psi_{j+1/2}^{(2l)}   + D_{j,2} - \frac{hD_{j,1}}{2}\right). \label{T_Taylor_expansion}
\end{equation}
Upon defining
\begin{equation}
    U_l:=h\sum_{j=\mu}^{\nu} \psi_{j+1/2}^{(2l)},\quad l\in \mathbb{N}_{\tilde{m}},\label{def_U_l}
\end{equation}
and 
\begin{equation}
    \eta_l:=\frac{2l}{4^{l}(2l+1)!},\quad l\in \mathbb{N}_{\tilde{m}}, \label{def_eta_sequence}
\end{equation}
we rewrite the equation \eqref{T_Taylor_expansion} as
\begin{equation}
    T = -\sum_{l=1}^{\tilde{m}}   \eta_l  h^{2l}U_l+\sum_{j=\mu}^{\nu}\left( D_{j,2} - \frac{hD_{j,1}}{2}\right). \label{T_Taylor_expansion_U_eta}
\end{equation}

We begin estimating $|T|$ by establishing the inequality
\begin{equation}
|T|\leqslant \left|T + \sum_{l=1}^{\tilde{m}}   \eta_l  h^{2l}U_l\right|+\left|\sum_{l=1}^{\tilde{m}} \eta_l h^{2l}U_l\right|. \label{est1}
\end{equation}
The next lemma estimates the first term of the right-hand side of \eqref{est1}.

\begin{lemma} \label{lem_estimation_T}
If $T$ is defined as in equation \eqref{def_T}, then
\begin{equation}
    \left|T+ \sum_{l=1}^{\tilde{m}}   \eta_l h^{2l}U_{l}\right|\leqslant 3(\nu-\mu+1)\|\psi^{(m)}\|_{L_\infty([s_\mu, s_{\nu+1}])}\left(\frac{h}{2}\right)^{m+1} .\label{form_T}
\end{equation}
\end{lemma}
\begin{proof}
By equation \eqref{T_Taylor_expansion_U_eta}  and the triangle inequality, we have that
\begin{equation}
    \left|T + \sum_{l=1}^{\tilde{m}}   \eta_l  h^{2l}U_l\right| \leqslant \sum_{j=\mu}^{\nu}\left(|D_{j,2}| + \frac{h|D_{j,1}|}{2}\right).\label{T_else}
\end{equation}
It suffices to estimate the reminder $D_{j,1}, D_{j,2}, j\in \mathbb{Z}_{[\mu,\nu]}$ of the Taylor expansion. Clearly, for each $j\in \mathbb{Z}_{[\mu,\nu]}$, we have the estimates
$$
|D_{j,1}| \leqslant \frac{h^m\|\psi^{(m)}\|_\infty}{2^{m-1} m!},
$$
where  $\|\psi^{(m)}\|_\infty:=\|\psi^{(m)}\|_{L_\infty([s_\mu, s_{\nu+1}])}$
and
$$
|D_{j,2}| \leqslant \frac{\|\psi^{(m)}\|_\infty}{m!}\int_{s_j}^{s_{j+1}}|t-s_{j+1/2}|^{m}\mathrm{d}t = \frac{h^{m+1}\|\psi^{(m)}\|_\infty}{2^m(m+1)!}.
$$
Substituting the above two estimates into the right-hand side of inequality \eqref{T_else} with noticing $\frac{m+2}{(m+1)!}\leqslant \frac{3}{2}$ yields the desired result. 
\end{proof}

We now consider the second term of the right-hand side of \eqref{est1}, which involves $U_l$ defined by \eqref{def_U_l}. Note that when $h=\frac{2}{p}$ and $\nu-\mu\leqslant p-1$, there holds
$$
|U_l|=\left|h\sum_{j=\mu}^{\nu} \psi_{j+1/2}^{(2l)}\right|\leqslant \frac{2(\nu-\mu+1)\|\psi^{(2l)}\|_\infty}{p}\leqslant 2\|\psi^{(2l)}\|_\infty,\quad l\in \mathbb{N}_{\tilde{m}}. 
$$
We will see later that $\|\psi^{(2l)}\|_\infty$ is in the order of $\kappa^{2l}$, which will lead to the conclusion  that the second term of the right-hand side of \eqref{est1} has a leading term in the order of $\kappa^{2}$. We would like to obtain a better estimate so that the leading term is in the order of $\kappa$, not $\kappa^2$. 
%
%
To this end, we need the Taylor expansion of the function $\psi^{(2l)}, l\in \mathbb{N}_{\tilde{m}}$, at the midpoint of the every subinterval. For each $j \in \mathbb{Z}_{[\mu,\nu]}$, and $l \in \mathbb{N}_{\tilde{m}}$, the Taylor Theorem ensures that for each $t \in [s_j, s_{j+1}]$, there exists a $\xi_{l,t}$ between $t$ and $s_{j+1/2}$ such that
\begin{equation}
    \psi^{(2l)}(t) = \psi_{j+1/2}^{(2l)} + \sum_{b=1}^{m-2l-1} \frac{(t-s_{j+1/2})^{(b)}}{b!}\psi^{(2l+b)}_{j+1/2} + \frac{(t-s_{j+1/2})^{m-2l}}{(m-2l)!} \psi^{(m)}(\xi_{l,t}). \label{another_Taylor}
\end{equation}
Integrating both sides of equation \eqref{another_Taylor} from $s_j$ to $s_{j+1}$ and defining
$$
R_{j,l} := \int_{s_j}^{s_{j+1}} \frac{(t-s_{j+1/2})^{m-2l}}{(m-2l)!} \psi^{(m)}(\xi_{l,t}) \mathrm{d}t,
$$ 
we obtain that
\begin{equation}\label{psi-l}
    \psi^{(2l-1)}(s_{j+1}) - \psi^{(2l-1)}(s_{j}) = h\psi_{j+1/2}^{(2l)} + \sum_{b=1}^{\tilde{m}-l} \frac{2}{(2b+1)!}\left(\frac{h}{2}\right)^{2b+1} \psi_{j+1/2}^{(2l+2b)} + R_{j,l}, \quad j \in \mathbb{Z}_{[\mu,\nu]}. 
\end{equation}
Summing up equation \eqref{psi-l} for $j \in \mathbb{Z}_{[\mu,\nu]}$, together with the definition \eqref{def_eta_sequence} of $\eta_l$,  we obtain that
\begin{equation*} 
\psi^{(2l-1)}(s_{\nu+1}) - \psi^{(2l-1)}(s_\mu) = h\sum_{j=\mu}^{\nu} \psi^{(2l)}_{j+1/2} + \sum_{j=\mu}^{\nu}\sum_{b=1}^{\tilde{m}-l} \frac{\eta_b}{2b} h^{2b+1} \psi_{j+1/2}^{(2l+2b)} + \sum_{j=\mu}^{\nu} R_{j,l}, \quad l \in \mathbb{N}_{\tilde{m}}. 
\end{equation*}
With the definition of $U_l, l\in \mathbb{N}_{\tilde{m}}$ in equation \eqref{def_U_l}, the above equation can be rewritten as 
\begin{equation}
    U_{l} + \sum_{b=1}^{\tilde{m}-l} \frac{\eta_b}{2b} h^{2b} U_{b+l}=\psi^{(2l-1)}(s_{\nu+1}) - \psi^{(2l-1)}(s_\mu)-\sum_{j=\mu}^{\nu} R_{j,l},\quad l \in \mathbb{N}_{\tilde{m}}. \label{important_relation}
\end{equation}

We now return to investigating the second term of the right-hand side of equation \eqref{est1}.

\begin{lemma} If
\begin{equation}
    \delta_l := \eta_l - \sum_{b=1}^{l-1} \frac{\eta_{l-b}\delta_{b}}{2l-2b}, \quad l \in \mathbb{N},
 \label{def_sequence_delta}
\end{equation}
then 
\begin{equation}
    \sum_{l=1}^{\tilde{m}} \eta_l h^{2l}U_{l}=\sum_{l=1}^{\tilde{m}} \delta_l h^{2l}\left(\psi^{(2l-1)}(s_{\nu+1}) - \psi^{(2l-1)}(s_\mu)-\sum_{j=\mu}^{\nu} R_{j,l}\right). \label{U_O_V}
\end{equation} 
\end{lemma}
\begin{proof}
According to equation \eqref{important_relation}, it suffices to prove that
\begin{equation} 
\sum_{l=1}^{\tilde{m}} \eta_l  h^{2l}U_{l}=\sum_{l=1}^{\tilde{m}} \delta_l h^{2l} \left(U_{l}+\sum_{b=1}^{\tilde{m}-l} \frac{\eta_b}{2b} h^{2b} U_{b+l}\right) . \label{U_O_V:equ1}
\end{equation}
By direct computation,  we observe that
\begin{equation} 
\sum_{l=1}^{\tilde{m}} \delta_l h^{2l} \left(U_{l}+\sum_{b=1}^{\tilde{m}-l} \frac{\eta_b}{2b} h^{2b} U_{b+l}\right) = \sum_{l=1}^{\tilde{m}} \left(\delta_l+\sum_{b=1}^{l-1} \frac{\eta_{l-b}\delta_b}{2l-2b}\right) h^{2l}U_{l}. \label{lem_estimation2_T:equ3}
\end{equation}
By the definition \eqref{def_sequence_delta} of $\delta_l$, $l\in \mathbb{N}$, we have that
$$
\eta_l=\delta_l+\sum_{b=1}^{l-1} \frac{\eta_{l-b}\delta_b}{2l-2b},\quad l\in \mathbb{N}.
$$
Substituting this into the right-hand side of equation \eqref{lem_estimation2_T:equ3} yields the desired equation \eqref{U_O_V:equ1}.
\end{proof}

Substituting equation \eqref{U_O_V} into the second term in the right-hand side of inequality \eqref{est1} yields that
\begin{equation}
    |T| \leqslant \left|T+ \sum_{l=1}^{\tilde{m}} \eta_l h^{2l}U_{l}\right| + \sum_{l=1}^{\tilde{m}} |\delta_l| h^{2l}\left(|\psi^{(2l-1)}(s_{\nu+1}) - \psi^{(2l-1)}(s_\mu)|+\sum_{j=\mu}^{\nu} |R_{j,l}|\right).  \label{main3_T}
\end{equation}
Substituting
\begin{equation*}
    |\psi^{(2l-1)}(s_{\nu+1}) - \psi^{(2l-1)}(s_\mu)| \leqslant 2 \|\psi^{(2l-1)}\|_\infty, 
\end{equation*}
$$
|R_{j,l}| \leqslant \frac{\|\psi^{(m)}\|_\infty}{(m-2l)!}\int_{s_j}^{s_{j+1}}|t-s_{j+1/2}|^{m-2l}\mathrm{d}t = \frac{h^{m-2l+1}\|\psi^{(m)}\|_\infty}{2^{m-2l}(m-2l+1)!} \leqslant 2\|\psi^{(m)}\|_\infty\left(\frac{h}{2}\right)^{m-2l+1}, 
$$
and estimate \eqref{form_T} into the right-hand side of inequality \eqref{main3_T}, we obtain that
\begin{equation}
    |T| \leqslant 2\sum_{l=1}^{\tilde{m}} h^{2l}|\delta_l| \|\psi^{(2l-1)}\|_\infty + (\nu-\mu+1)\left(3+2\sum_{l=1}^{\tilde{m}}4^l|\delta_l|\right) \|\psi^{(m)}\|_\infty\left(\frac{h}{2}\right)^{m+1}. \label{lem_estimation2_main_T:equ1}
\end{equation}

By the above inequality, we estimate the growth of derivatives of function $\psi$ in the following lemma.

\begin{lemma} \label{lem_derivative_of_T}
If $\psi$ is defined in equation \eqref{def_psi}, then
\begin{equation}
    \| \psi^{(l)}\|_\infty\leqslant \tau [(\Gamma+2)\kappa+1]^l,\quad l\in \mathbb{Z}_{m+1}, \label{psi_dev_estimation}
\end{equation}
where $\|\psi^{(m)}\|_\infty:=\|\psi^{(m)}\|_{L_\infty([s_\mu, s_{\nu+1}])}$.
\end{lemma}
\begin{proof}
Without loss of generality, we will assume that $0\leqslant \mu \leqslant \nu\leqslant j^*-1$. In this case, the definition \eqref{def_psi} of $\psi$ yields that
$\psi(t)=w(t)e^{i(\alpha-1)\kappa t}$, $t\in [s_\mu, s_{\nu+1}]$.
Repeatedly differentiating $\psi$ yields
\begin{equation}
    \psi^{(l)}(s) = e^{i(\alpha-1) \kappa s}\sum_{d=0}^{l} C_{l,d} w^{(d)}(s), \quad s\in [s_\mu, s_\nu], \quad l\in \mathbb{Z}_{m+1}, \label{lem_h1_kappa_equ1}
\end{equation}
where $C_{0,0}:=1$ and  for $l\in \mathbb{N}_{m}$,
$$
C_{l, d}:=\begin{cases}
i (\alpha-1)\kappa C_{l-1,d}, &  d=0,\\
i(\alpha-1)\kappa C_{l-1,d}+C_{l-1,d-1}, &  d = 1,2,\dots, l-1,\\
C_{l-1,d-1}, &  d=l.
\end{cases}
$$
It can be shown by induction on $l$ that
\begin{equation}
    \sum_{d=0}^{l}|C_{l, d}| \leqslant (1+|\alpha-1|\kappa)^l, \quad l\in \mathbb{Z}_{m+1}. \label{lem_h1_kappa_equ2}
\end{equation}
Now, by applying the $L_\infty$-norm of $\psi^{(l)}$ to equation \eqref{lem_h1_kappa_equ1}, we obtain that
\begin{equation}
    \|\psi^{(l)}\|_\infty  \leqslant \sum_{d=0}^{l}|C_{l,d}| \|w^{(d)}\|_\infty, \quad l\in \mathbb{Z}_{m+1}. \label{lem_h1_kappa_equ3}
\end{equation}
Substituting \eqref{requirement} and  \eqref{lem_h1_kappa_equ2} into the right-hand side of \eqref{lem_h1_kappa_equ3} yields 
$$
\| \psi^{(l)}\|_\infty\leqslant \tau (1+|\alpha-1|\kappa)^l,\quad l\in \mathbb{Z}_{m+1}.   
$$
The above inequality with $|\alpha-1|\leqslant |\alpha|+1\leqslant \Gamma+2$ leads to the desired estimate.
\end{proof}

We next estimate $|T|$. To this end, we define $\Delta(\tilde m):=\sum_{l=1}^{\tilde{m}}4^l|\delta_l|$.

\begin{lemma} \label{lem_estimation_main_T}
If $p\in \mathbb{N}$ is chosen to satisfy $p\geqslant (\Gamma+2)\kappa+1$, then
\begin{equation}
    |T| \leqslant \frac{2\tau ((\Gamma+2)\kappa+1)}{p^2}\Delta(\tilde m)+\left(3+2\Delta(\tilde m)\right)\frac{\tau }{p} \left(\frac{(\Gamma+2)\kappa+1}{p}\right)^m.  \label{main_T}
\end{equation}
\end{lemma}

\begin{proof}
We prove this lemma by employing inequality \eqref{lem_estimation2_main_T:equ1}.
As for the first term in the right-hand side of inequality \eqref{lem_estimation2_main_T:equ1}, by inequality \eqref{psi_dev_estimation} in Lemma \ref{lem_derivative_of_T} and the fact $h=2/p$, with noting $p\geqslant (\Gamma+2)\kappa+1$, we observe that
\begin{flalign*}
    2\sum_{l=1}^{\tilde{m}} h^{2l}|\delta_l| \|\psi^{(2l-1)}\|_\infty&\leqslant \frac{2\tau ((\Gamma+2)\kappa+1)}{p^2}\sum_{l=1}^{\tilde{m}}4^l|\delta_l|\left(\frac{(\Gamma+2)\kappa+1}{p}\right)^{2l-2}\\
    &\leqslant \frac{2\tau ((\Gamma+2)\kappa+1)}{p^2}\Delta(\tilde{m}).   
\end{flalign*}
As for the second term in the right-hand side of inequality \eqref{lem_estimation2_main_T:equ1}, for the same reason, we obtain that
\begin{flalign*}
    \left(3+2\Delta(\tilde m)\right) \|\psi^{(m)}\|_\infty\left(\frac{h}{2}\right)^{m+1}&\leqslant \left(3+2\Delta(\tilde m)\right)\frac{\tau }{p} \left(\frac{(\Gamma+2)\kappa+1}{p}\right)^m.   
\end{flalign*}
Substituting the above two inequalities into the right-hand side of inequality \eqref{lem_estimation2_main_T:equ1} yields the desired estimate \eqref{main_T}.
\end{proof}

Note that $\Delta(\tilde m)\leqslant\sum_{l=1}^\infty 4^l|\delta_l|$, which is estimated in the next lemma.

\begin{lemma} \label{lem_seq}
If $\delta_l$,  $l \in \mathbb{N}$, is the sequence defined in \eqref{def_sequence_delta}, then
\begin{equation*}
    \sum_{l=1}^\infty 4^l|\delta_l| \leqslant \frac{6}{5}. 
\end{equation*}
\end{lemma}

\begin{proof}
Defining
\begin{equation}
    \tilde{\delta}_l := \frac{1}{(2l)!} + \sum_{b=1}^{l-1} \frac{\tilde{\delta}_{l-b}}{(2b)!}, \quad l \in \mathbb{N}, \label{def_sequence_tildedelta}
\end{equation}
this lemma may be proved by establishing the two estimates
\begin{equation}
    4^l|\delta_l| \leqslant \tilde{\delta}_l, \quad l \in \mathbb{N} \label{lem_seq_to_prove1}
\end{equation}
and 
\begin{equation}
    \sum_{l=1}^{\infty} \tilde{\delta_l} \leqslant \frac{6}{5}. \label{lem_seq_to_prove2}
\end{equation}

First, we establish \eqref{lem_seq_to_prove1} via induction on $l \in \mathbb{N}$. By definition \eqref{def_sequence_delta}, the sequence $\delta_l$, $l \in \mathbb{N}$,  can be written as
\begin{equation}
\delta_{l} = \eta_{l} - \sum_{b=1}^{l-1} \frac{\delta_{l-b}\eta_{b}}{2b},\quad l \in \mathbb{N}. \label{another_form_delta}
\end{equation}
For $l = 1$, it is clear that $4|\delta_1| = \frac{1}{3} \leqslant \tilde{\delta}_1 = \frac{1}{2}$. Now, we assume for some $l^* \in \mathbb{N}$ that
\begin{equation}
    4^b|\delta_{b}| \leqslant \tilde{\delta}_{b},\quad b \in \mathbb{N}_{l^*} \label{assumation}
\end{equation}
and consider the case $l = l^* + 1$. Utilizing equation \eqref{another_form_delta} with $l := l^* + 1$ and the fact $4^b\eta_b = \frac{2b}{(2b+1)!} \leqslant \frac{1}{(2b)!}$, $b \in \mathbb{N}$, we derive that
\begin{align*}
    4^{l^*+1}|\delta_{l^*+1}| \leqslant \frac{1}{(2l^*+2)!} + \sum_{b=1}^{l^*}  \frac{4^{l^*+1-b}|\delta_{l^*+1-b}|}{(2b)!}.
\end{align*}
Using \eqref{assumation}, the above inequality becomes
$$ 
|4^{l^*+1}\delta_{l^*+1}| \leqslant \frac{1}{(2l^*+2)!} + \sum_{b=1}^{l^*}  \frac{\tilde{\delta}_{l^*+1-b}}{(2b)!}, 
$$
whose right-hand side is exactly equal to  $\tilde{\delta}_{l^*+1}$ by definition \eqref{def_sequence_tildedelta}. Thus, we have proved  inequality \eqref{lem_seq_to_prove1} for the case $l := l^* + 1$. By the induction principle,  inequality \eqref{lem_seq_to_prove1} holds for all $l\in \mathbb{N}$.

We now prove  \eqref{lem_seq_to_prove2}.
By defining a positive monotonically increasing sequence 
\begin{equation}
    \Delta_l := \sum_{b=1}^l \tilde{\delta}_b, \quad l \in \mathbb{N}, \label{def_sequence_Delta}
\end{equation}
it suffices to show
\begin{equation}
    \lim_{l\to \infty} \Delta_l \leqslant \frac{6}{5}. \label{lem_seq_equ0}
\end{equation}
We first derive a recursive formula for the sequence $\Delta_l$, $l \in \mathbb{N}$. To this end, we substitute definition \eqref{def_sequence_tildedelta} of $\tilde{\delta}_l$ into definition \eqref{def_sequence_Delta} of $\Delta_l$ to obtain that
\begin{equation}
    \Delta_l = \sum_{b=1}^l \tilde{\delta}_b = \sum_{b=1}^l \frac{1}{(2b)!} + \sum_{d=1}^l \sum_{b=1}^{d-1} \frac{\tilde{\delta}_{d-b}}{(2b)!}, \quad l \in \mathbb{N}. \label{lem_seq_equ01}
\end{equation}
Since for all $l \in \mathbb{N}$,
$$
\sum_{d=1}^l \sum_{b=1}^{d-1} \frac{\tilde{\delta}_{d-b}}{(2b)!} =\sum_{b=1}^{l-1} \sum_{d=b+1}^{l} \frac{\tilde{\delta}_{d-b}}{(2b)!}= \sum_{b=1}^{l-1}\frac{\Delta_{l-b}}{(2b)!}, 
$$
where the last equality holds due to definition \eqref{def_sequence_Delta} of $\Delta_l$, equation \eqref{lem_seq_equ01} becomes a recursion of $\Delta_l$:
\begin{equation}
    \Delta_l = \sum_{b=1}^l \frac{1}{(2b)!} + \sum_{b=1}^{l-1}\frac{\Delta_{l-b}}{(2b)!}, \quad l \in \mathbb{N}. \label{lem_seq_equ1}
\end{equation}
We observe the sequence $\Delta_l$, $l \in \mathbb{N}$ is monotonically increasing. In particular, for $l \geqslant 2$, we have that $\Delta_b \leqslant \Delta_l$ for all $b \in \mathbb{N}_{l-1}$. Consequently, equation \eqref{lem_seq_equ1} implies that
\begin{equation*}
    \Delta_l \leqslant \sum_{b=1}^l \frac{1}{(2b)!} + \Delta_l\sum_{b=1}^{l-1}\frac{1}{(2b)!}, \quad l \geqslant 2.
\end{equation*}
In light of the fact that $\Delta_l > 0$ for $l \in \mathbb{N}$ and  $\sum_{b=1}^\infty \frac{1}{(2b)!} = \cosh(1)-1$, we deduce
\begin{equation*}
    \Delta_l \leqslant (\cosh(1)-1)(\Delta_l+1), \quad l \geqslant 2,
\end{equation*}
namely, 
$$\Delta_l \leqslant \frac{\cosh(1)-1}{2-\cosh(1)}\leqslant \frac{6}{5}, \quad l \geqslant 2.$$
Thus, by the Monotone Convergence Theorem, we conclude the existence of the limit for sequence $\Delta_l$, $l \in \mathbb{N}$, and confirm inequality \eqref{lem_seq_equ0}.
\end{proof}

Combining Lemmas \ref{lem_estimation_main_T} and \ref{lem_seq} yields the following estimate of $|T|$.

\begin{lemma} \label{lem_estimation2_main_T}
If $p\in \mathbb{N}$ is chosen to satisfy $p\geqslant (\Gamma+2)\kappa+1$, then
\begin{equation}
    |T| \leqslant \frac{12\tau ((\Gamma+2)\kappa+1)}{5p^2}+\frac{27(\nu-\mu+1) \tau}{5p} \left(\frac{(\Gamma+2)\kappa+1}{p}\right)^m.  \label{main2_T}
\end{equation}
\end{lemma}

Finally, we have the next proposition for an estimate of $\|\mathcal{K}\chi_{\kappa}-\mathcal{K}_{p} \chi_{\kappa}\|_\infty$.

\begin{proposition}  \label{prop_estimation_chi}
For a complex-valued analytic $\chi_{\kappa}$ as defined in \eqref{def_chi}, for any $\kappa \geqslant 1$, if  parameter $p\in \mathbb{N}$ is selected to satisfy $p\geqslant (\Gamma+2)\kappa+1$, then
    \begin{equation*}
        \|\mathcal{K}\chi_{\kappa}-\mathcal{K}_{p} \chi_{\kappa}\|_\infty
    \leqslant \frac{44r\tau}{5p}+\frac{27r \tau }{5} \left(\frac{(\Gamma+2)\kappa+1}{p}\right)^m.
    \end{equation*}
\end{proposition} 
\begin{proof}
It suffices to prove for the $\tilde{\chi}_{\kappa}$ as defined in \eqref{def_tilde_chi}, it holds that
\begin{equation}
\|\mathcal{K}\tilde{\chi}_{\kappa}-\mathcal{K}_{p} \tilde{\chi}_{\kappa}\|_\infty \leqslant \frac{44\tau}{5p}+\frac{27\tau}{5}\left(\frac{(\Gamma+2)\kappa+1}{p}\right)^m. \label{lem_estimation_tilde_chi_target}
\end{equation}
By letting $\mu:=0,\nu:=j^*-1$ and $\mu:=j^*+1, \nu:=p-1$ in the definition of $T$ in equation \eqref{def_T}, the inequality \eqref{main2_T} in Lemma  \ref{lem_estimation2_main_T} holds for $|T|=|T_1|$ and $|T|=|T_3|$ respectively, namely
     $$ |T_1| \leqslant \frac{12\tau ((\Gamma+2)\kappa+1)}{5p^2}+\frac{27j^* \tau}{5p} \left(\frac{(\Gamma+2)\kappa+1}{p}\right)^m$$
     and
     $$|T_3| \leqslant \frac{12\tau ((\Gamma+2)\kappa+1)}{5p^2}+\frac{27(p-j^*-1) \tau}{5p} \left(\frac{(\Gamma+2)\kappa+1}{p}\right)^m.$$
     Together with the estimate of $|T_2|$ in Lemma \ref{lem_T2} and the fact $j^*+(p-j^*-1)\leqslant p$, the inequality \eqref{estimate_T1_T2_T3} implies that
     $$|(\mathcal{K}\tilde{\chi}_{\kappa}-\mathcal{K}_{p} \tilde{\chi}_{\kappa})(s)| \leqslant \frac{4\tau}{p}\left(1+\frac{6 ((\Gamma+2)\kappa+1)}{5p}\right)+\frac{27 \tau}{5} \left(\frac{(\Gamma+2)\kappa+1}{p}\right)^m.$$
     The above equality holds for any fixed $s\in I$, with noting $p\geqslant (\Gamma+2)\kappa+1$, yielding the estimate \eqref{lem_estimation_tilde_chi_target}.

\end{proof}

Proposition \ref{prop_estimation_chi} elucidates the connection among the error, the sample size $p$ and the wavenumber $\kappa$. Next, we aim at ascertaining the value of $p$ as a function of the wavenumber $\kappa$ such that the error tends to zero as $\kappa$ approaches infinity. This will be expounded upon in the following proposition. Before this, for any $x\in \mathbb{R}$, $\lceil x \rceil$ denotes the smallest integer greater than or equal to  $x$

\begin{proposition} \label{prop_choose_parameter}
    For the function $\chi_{\kappa}$ as defined in \eqref{def_chi}, suppose the parameters $\beta$ and $\gamma$ are chosen to satisfy
$\beta \geqslant 1$ and $\gamma \geqslant \Gamma+3$.
If $p_{_\kappa} := \lceil \gamma \kappa^\beta \rceil$, then for any $\kappa \geqslant 1$, 
\begin{equation*}
        \|\mathcal{K}\chi_{\kappa} - \mathcal{K}_{p_{\kappa}} \chi_{\kappa}\|_\infty \leqslant \frac{44 r\tau}{5\gamma \kappa^\beta}+\frac{27r\tau(\Gamma+3)^m}{5\gamma^m \kappa^{m(\beta-1)}}. \label{prop_choose_parameter_target}
\end{equation*}
\end{proposition}

\begin{proof}
 First, we need to verity $p_{_{\kappa}}\geqslant (\Gamma+2)\kappa+1$ which is the condition of Lemma \ref{prop_estimation_chi}. For $\beta \geqslant 1$,  we consider $\frac{\kappa^{\beta-1} ((\Gamma+2)\kappa+1)}{ p_{_\kappa}}, \kappa\geqslant 1$. By the definition of $p_{_\kappa}$, it holds that 
 $$\frac{\kappa^{\beta-1} ((\Gamma+2)\kappa+1)}{ p_{_\kappa}}\leqslant \frac{(\Gamma+2)\kappa+1}{\gamma\kappa}, \quad \kappa\geqslant 1.$$
 Noting that the right hand side of above inequality equal to $\frac{\Gamma+2}{\gamma}+\frac{1}{\gamma\kappa}$ is a decreasing function for $\kappa\geqslant 1$, the above inequality implies 
   \begin{equation*}
\frac{\kappa^{\beta-1} ((\Gamma+2)\kappa+1)}{ p_{_\kappa}}\leqslant \frac{\Gamma+3}{\gamma}, \quad \kappa\geqslant 1.
\end{equation*}
This is,
\begin{equation}
\frac{(\Gamma+2)\kappa+1}{ p_{_\kappa}}\leqslant \frac{\Gamma+3}{\gamma \kappa^{\beta-1}}, \quad \kappa\geqslant 1.
\label{F_p_relation_ex}
\end{equation}
Since $\gamma\geqslant \Gamma+3$, the above equation implies that $p_{_\kappa}\geqslant (\Gamma+2)\kappa+1$. Hence, by Lemma \ref{prop_estimation_chi} with $p :=p_{_\kappa}$, we obtain that
  \begin{equation*}
        \|\mathcal{K}\chi_{\kappa} - \mathcal{K}_{p_{\kappa}} \chi_{\kappa}\|_\infty \leqslant \frac{44r\tau }{5 p_{_\kappa}}+\frac{27r \tau }{5} \left(\frac{(\Gamma+2)\kappa+1}{p_{_\kappa}}\right)^m. 
    \end{equation*}
Substituting $p_{_\kappa}\geqslant \gamma \kappa^\beta$ and inequality \eqref{F_p_relation_ex} into the two terms in the right hand side of above inequality respectively yields the desired result directly.
  \end{proof}
  
As for $y \in H^m_{\kappa, 0}(I)$ as shown in \eqref{solution}, we have the next proposition.

\begin{proposition}
\label{proposition_integral_y}
Suppose that the function $y$ has the form \eqref{solution} and $\Gamma \geqslant 0$ is a given relaxation factor. If the parameters $\beta$ and $\gamma$ are chosen to satisfy
$\beta \geqslant 1$,  $\gamma \geqslant \Gamma+3$
and $p_{_\kappa} := \lceil \gamma \kappa^\beta \rceil$, then for any $\kappa \geqslant 1$, 
\begin{equation}
\|\mathcal{K}y-\mathcal{K}_{p_{_\kappa}} y\|_\infty \leqslant   \frac{132\tau}{5\gamma \kappa^\beta}+\frac{81\tau(\Gamma+3)^m}{5\gamma^m \kappa^{m(\beta-1)}}. \label{integral_error}
\end{equation}
\end{proposition}

\begin{proof}
Note that the function $y$ having the form \eqref{solution} can be expressed in the form $\chi_{\kappa}$ by letting $r := 3$, $\alpha_1 := 0$, $\alpha_2 := -1$, $\alpha_3 := 1,$ and $w_j := u_j$ for $j\in\mathbb{N}_3$. Then, the equation \eqref{integral_error} can be obtained by setting $\chi_{\kappa} = y$ in Proposition \ref{prop_choose_parameter}.
\end{proof}

For the purpose of approximating the operator $\mathcal{K}$, by  Proposition \ref{proposition_integral_y}, we propose the discrete oscillatory integral operator $\mathcal{K}_{p_{_\kappa}}$ by taking 
 \begin{equation}
     p_{_\kappa}:=\lceil \gamma \kappa^\beta\rceil, \label{def_p_kappa}
 \end{equation}
 where the parameters $\beta$ and $\gamma$ are chosen to satisfy
\begin{equation}
 \beta\geqslant 1, \ \ \gamma \geqslant \Gamma+3. \label{rule_beta_gamma}
\end{equation}  
 Up to now, with the the discrete oscillatory integral operator $\mathcal{K}_{p_{_\kappa}}$, the DNN learning model is well described.

In the next section, we will estimate the error of the solution obtained from the DNN model \eqref{DNN_opt}. To distinguish it from a MGDL model to be described in section 6, we will refer to the standard DNN model \eqref{DNN_opt} as the single-grade learning model.

\section{Analysis of Single-Grade Learning Model}
The purpose of this section is to bound the error of the DNN approximate solution by the training loss and the quadrature error.

Throughout this section, we assume that the solution $y$ has the form \eqref{solution}, the parameters $\beta$ and $\gamma$ are chosen according to rule \eqref{rule_beta_gamma}, and $p_{_\kappa}$ is defined as in equation \eqref{def_p_kappa}. The operator $\mathcal{K}_{p_{\kappa}}$ is defined by setting $p:=p_{_\kappa}$ as described in equation \eqref{def_mathcal_K_p}.


We first derive an equivalent form of the fully discrete minimization problem \eqref{DNN_opt}.
To this end, for each $\kappa\geqslant 1$ and a fixed $q\in \mathbb{N}$, we let
$N_\kappa:=qp_{_\kappa}+1$
and 
\begin{equation}
      \omega_\kappa:=e^{\frac{2\kappa}{qp_{_\kappa} }i}.\label{def_omega_kappa}
  \end{equation}
We then define the matrix 
\begin{equation}
    \mathbf{B}_{\kappa}:=(b_{ j,l}(\kappa))_{j,l\in \mathbb{N}_{N_\kappa}} \in \mathbb{C}^{N_\kappa\times N_\kappa}\label{def_B_kappa}
\end{equation}
with
$$
b_{j,l}(\kappa):=\begin{cases}
      \omega_\kappa^{|j-l|}, & l\in \{1, N_\kappa\},\\
      2\omega_\kappa^{|j-l|}, & l\in \{dq+1: d\in \mathbb{N}_{p_{_\kappa} -1}\},\\
      0, &\mathrm{otherwise},
  \end{cases}\ \ \mbox{for}\ \ j\in \mathbb{N}_{N_\kappa}.
$$
Since $p_{_\kappa}=\lceil \gamma\kappa^\beta \rceil$ where $\gamma, \beta$ is chosen by the rule \eqref{rule_beta_gamma}, the matrix $\mathbf{B}_{\kappa}$ is completely determined by four independent parameters $\kappa, \beta, \gamma, q$. However in this paper, we only concern the influence of wavenumber $\kappa$. Thus, notation $\mathbf{B}_\kappa$ indicates only its dependence on $\kappa$ even though it depends on the other three parameters $\beta, \gamma, q$, and so do $p_{_\kappa}$ and $N_\kappa$. 

We assume that the training data $\{(x_j, f(x_j)), j\in \mathbb{N}_{N_\kappa}\}$ are chosen as 
$x_j:=-1+\frac{2(j-1)}{N_\kappa-1}$, $j\in \mathbb{N}_{N_\kappa}$.
Thus, for each $h\in C(I)$,  the $ [((\mathcal{I}-\lambda\mathcal{K}_{p_{_\kappa}})h)(x_j): j=1,2, \dots, N_\kappa]^T$ has an equivalent matrix form. We show it in the next lemma.

\begin{lemma} \label{lem_space_difference}
For all $\kappa\geqslant 1$, $h\in C(I)$, there holds
\begin{equation}
    [((\mathcal{I}-\lambda\mathcal{K}_{p_{_\kappa}})h)(x_j): j=1,2, \dots, N_\kappa]^T=\left(\mathbf{I}_\kappa-\frac{\lambda }{p_{_\kappa}}\mathbf{B}_\kappa \right)\mathbf{v}_h,\label{lem_space_difference:target2-new}
\end{equation}
where $\mathbf{v}_h:=[h(x_j): j\in \mathbb{N}_{N_\kappa}]$ and $\mathbf{I}_\kappa$ denotes the $N_\kappa\times N_\kappa$  identity matrix.
\end{lemma}
\begin{proof}
For an $h\in C(I)$, by the definition of operator $\mathcal{I}-\lambda\mathcal{K}_{p_{_\kappa}}$, for each $j\in \mathbb{N}_{N_\kappa}$,  $((\mathcal{I}-\lambda\mathcal{K}_{p_{_\kappa}})h)(x_j)$ can be written as 
$$
H_j:= h(x_j)-\frac{\lambda}{p_{_\kappa}}\left(h(x_1)e^{i\kappa |x_1-x_j|}+2\sum_{l=1}^{p_{_\kappa}-1} h(x_{ql+1})e^{i\kappa |x_j-x_{ql+1}|}+h(x_{N_\kappa}) e^{i\kappa|x_j-x_{N_\kappa}|}\right).
$$
Using the definition \eqref{def_omega_kappa} of $\omega_\kappa$, we may further rewrite $H_j$ as 
\begin{align*}
H_j = h(x_j)-\frac{\lambda}{p_{_\kappa}}\left(\omega_\kappa^{|1-j|}h(x_1)+2\sum_{l=1}^{p_{_\kappa}-1} \omega_\kappa^{|j-ql-1|}h(x_{ql+1})+ \omega_\kappa^{|j-N_\kappa|}h(x_{N_\kappa})\right),\ \ j\in \mathbb{N}_{N_\kappa}. 
\end{align*}
It can be verified that the right-hand side of the above equation is the $j$-th component of the vector $(\mathbf{I}_\kappa-\frac{\lambda}{p_{_\kappa}}\mathbf{B}_\kappa)\mathbf{v}_h$. This proves the desired result. 
\end{proof}

Lemma \ref{lem_space_difference} relates the operator $\mathcal{I}-\lambda \mathcal{K}_{p_{_\kappa}}$ with the matrix 
\begin{equation}
    \mathbf{M}_\kappa:= \mathbf{I}_\kappa-\lambda\mathbf{B}_\kappa/p_{_\kappa}\in \mathbb{C}^{N_\kappa\times N_\kappa}. \label{def_M_kappa}
\end{equation}
In equation \eqref{lem_space_difference:target2-new}
by choosing $h:=\mathcal{T}\mathcal{N}_n(\{\mathbf{W}_l, \mathbf{b}_l\}_{l=1}^n;\cdot)$, we obtain that
$$
[((\mathcal{I}-\lambda\mathcal{K}_{p_{_\kappa}})\mathcal{T}\mathcal{N}_n(\{\mathbf{W}_l, \mathbf{b}_l\}_{l=1}^n;\cdot))(x_j): j=1,2, \dots, N_\kappa]^T=\mathbf{M}_\kappa \mathbf{v}_g (\{\mathbf{W}_l, \mathbf{b}_l\}_{l=1}^n),
$$
where
$$
\mathbf{v}_g(\{\mathbf{W}_l, \mathbf{b}_l\}_{l=1}^n):=[\mathcal{T}\mathcal{N}_n(\{\mathbf{W}_l, \mathbf{b}_l\}_{l=1}^n;x_l), l\in \mathbb{N}_{N_\kappa}]\in \mathbb{C}^{N_\kappa}.
$$
Moreover, by definition \eqref{single_grade_error_function} of $\tilde{e}$ and by letting 
$\mathbf{v}_f:=[f(x_l): l\in \mathbb{N}_{N_\kappa}]\in \mathbb{C}^{N_\kappa}$, it holds that
\begin{equation*}
    [\tilde{e}(x_j): j=1,2, \dots, N_\kappa]^T=\mathbf{v}_f-\mathbf{M}_\kappa \mathbf{v}_g (\{\mathbf{W}_l, \mathbf{b}_l\}_{l=1}^n).   
\end{equation*}
Thus, the minimization problem \eqref{DNN_opt} is equivalent to the following minimization problem
\begin{equation}
    \arg\min_{\{\mathbf{W}_l, \mathbf{b}_l\}_{l=1}^n} \frac{1}{N_\kappa}\left\|\mathbf{M}_\kappa \mathbf{v}_g (\{\mathbf{W}_l, \mathbf{b}_l\}_{l=1}^n)-\mathbf{v}_f\right\|^2_{l^2}. \label{opt_least_square}
\end{equation}
The optimization problem \eqref{opt_least_square} may be solved by the Adam (Adaptive Moment Estimation) optimization algorithm \cite{kingma2014adam} with the initialization proposed in \cite{he2015delving}. Upon finding the parameters $\{\tilde{\mathbf{W}}^*_j, \tilde{\mathbf{b}}^*_j\}_{j=1}^n$, the numerical solution $\tilde{y}^*$ of equation \eqref{fredholm_equation_operator} and the related error $\tilde{e}^*$ are given by \eqref{def_fin_y_star} and \eqref{def_fin_e}, respectively.

In the rest of this section, we will estimate the error $\|y-\tilde{y}^*\|_{N_\kappa}$ in terms of the loss error $\|\tilde{e}^*\|_{N_\kappa}$ and the quadrature error $\|(\mathcal{K}-\mathcal{K}_{p_{_\kappa}})y\|_\infty$.
We first estimate the error $\|(\mathcal{I}-\lambda \mathcal{K}_{p_{_\kappa}})(y-\tilde{y}^*)\|_{N_\kappa}$.

\begin{lemma} \label{lem_err_before_p}
For each $\kappa\geqslant 1$, there holds that
\begin{equation}
    \|(\mathcal{I}-\lambda \mathcal{K}_{p_{_\kappa}})(y-\tilde{y}^*)\|_{N_\kappa} \leqslant \|\tilde{e}^*\|_{N_\kappa}+|\lambda|\|(\mathcal{K}-\mathcal{K}_{p_{_\kappa}})y\|_\infty.\label{lem_err_before_p:target}
\end{equation}
\end{lemma}
\begin{proof}
By the definition \eqref{def_fin_e}  of $\tilde{e}^*$, we have that
$\tilde{e}^*= f- (\mathcal{I}-\lambda\mathcal{K}_{p_{_\kappa}})\tilde{y}^*$, which
together with $f=(\mathcal{I}-\lambda\mathcal{K})y$ yields that
$$
\tilde{e}^*= (\mathcal{I}-\lambda\mathcal{K})y-(\mathcal{I}-\lambda\mathcal{K}_{p_{_\kappa}})\tilde{y}^*.
$$
Thus, we obtain that
\begin{align*}
    (\mathcal{I}-\lambda\mathcal{K}_{p_{_\kappa}})(y-\tilde{y}^*)&=(\mathcal{I}-\lambda\mathcal{K})y-(\mathcal{I}-\lambda\mathcal{K}_{p_{_\kappa}})\tilde{y}^*+\lambda(\mathcal{K}-\mathcal{K}_{p_{_\kappa}})y\\
    &=\tilde{e}^* +\lambda(\mathcal{K}-\mathcal{K}_{p_{_\kappa}})y.
\end{align*}
By the triangle inequality of the semi-norm $\|\cdot\|_{N_\kappa}$, we find that
\begin{equation}
    \|(\mathcal{I}-\lambda\mathcal{K}_{p_{_\kappa}})(y-\tilde{y}^*)\|_{N_\kappa} \leqslant  \|\tilde{e}^*\|_{N_\kappa}+|\lambda| \|(\mathcal{K}-\mathcal{K}_{p_{_\kappa}})y\|_{N_\kappa}. \label{lem_err_before_p:equ1}
\end{equation}
The definition of $\|\cdot\|_{N_\kappa}$ yields that
$$
\|(\mathcal{K}-\mathcal{K}_{p_{_\kappa}})y\|_{N_\kappa}=\sqrt{\frac{1}{{N_\kappa}}\sum_{j=1}^{N_\kappa} |(\mathcal{K}y-\mathcal{K}_{p_{_\kappa}}y)(x_j)|^2}\leqslant  \|(\mathcal{K}-\mathcal{K}_{p_{_\kappa}})y\|_\infty.
$$
Substituting this inequality into the right-hand side of \eqref{lem_err_before_p:equ1}, we obtain the desired result \eqref{lem_err_before_p:target}.
\end{proof}

Next, we provide a condition that ensures the invertibility of the matrix $\mathbf{M}_\kappa$.

\begin{lemma} \label{prop_C_kappa}
If there exists a constant $C_\kappa>0$ such that for any $h\in C(I)$, 
\begin{equation}
    \|h\|_{N_\kappa}\leqslant C_\kappa \|(\mathcal{I}-\lambda\mathcal{K}_{p_{_\kappa}})h\|_{N_\kappa}, \label{prop_C_kappa:target}
\end{equation}
then the matrix $\mathbf{M}_\kappa$ is invertible.
\end{lemma}
\begin{proof}
We establish the invertibility of the matrix $\mathbf{M}_\kappa$ by contradiction. Assume, to the contrary, that $\mathbf{M}_\kappa$ is not invertible. Then, there exists a nonzero vector $\mathbf{v} := [v_j, j \in \mathbb{N}_\kappa]^T \in \mathbb{C}^{N_\kappa}$ such that $\mathbf{M}_\kappa \mathbf{v} = \mathbf{0}$. By the Lagrange interpolation formula, we can construct a polynomial $h_\mathbf{v} \in C(I)$ such that $h_{\mathbf{v}}(x_j) = v_j$, $j \in \mathbb{N}_{N_\kappa}$. By letting $h := h_\mathbf{v}$ in inequality \eqref{prop_C_kappa:target}, we obtain that
\begin{equation}
    \|h_{\mathbf{v}}\|_{N_{\kappa}} \leqslant C_\kappa\|(\mathcal{I} - \lambda\mathcal{K}_{p_{_\kappa}})h_\mathbf{v}\|_{N_\kappa}.  \label{prop_C_kappa:equ1}
\end{equation}
Noting 
\begin{equation*}
    \|h_{\mathbf{v}}\|_{N_{\kappa}} = \|\mathbf{v}\|_{l_2}/\sqrt{N_\kappa},\ \ \mbox{and}\ \  \|(\mathcal{I}-\lambda\mathcal{K}_{p_{_\kappa}})h_\mathbf{v}\|_{N_\kappa}=\left.\left\|\mathbf{M}_\kappa \mathbf{v}\right\|_{l_2}\right/\sqrt{N_\kappa}
\end{equation*}
by relations \eqref{relation_N_2} and \eqref{def_M_kappa} respectively, the inequality \eqref{prop_C_kappa:equ1} implies that
\begin{equation} 
    \|\mathbf{v}\|_{l_2} \leqslant C_\kappa \left\|\mathbf{M}_\kappa \mathbf{v}\right\|_{l_2}. \label{prop_C_kappa:equ2}
\end{equation}
However, note that the vector $\mathbf{v}$ is nonzero and $\mathbf{M}_\kappa \mathbf{v} = \mathbf{0}$, which contradicts inequality \eqref{prop_C_kappa:equ2}. This concludes that the matrix $\mathbf{M}_\kappa$ is invertible.
\end{proof}

Our next step is to identify the range of $\kappa$ that ensures the invertibility of matrix $\mathbf{M}_\kappa$. Note that matrix $\mathbf{M}_\kappa$ is related to matrix $\mathbf{B}_\kappa$ by relation \eqref{def_M_kappa} and many columns of matrix $\mathbf{B}_\kappa$ as defined in equation \eqref{def_B_kappa} are zero. We denote the set of the indices of the zero columns of $\mathbf{B}_\kappa$ by 
$J := \mathbb{N}_{N_\kappa} \setminus \{ql+1 : l \in \mathbb{Z}_{p_{\kappa+1}}\}$. 
Hence, by relation \eqref{def_M_kappa}, for any $j\in J$, the entries in the $j$-th column of matrix $\mathbf{M}_\kappa$ are all zeros but the $j$-th entry, which is $1$.  Combining this observation with the fact that a square matrix is invertible if and only if its determinant is non-zero, we will expand the determinant of matrix $\mathbf{M}_\kappa$ corresponding to the indices in set $J$,  
thereby obtaining the range of $\kappa$. 

Specifically, assuming that $j_1$ is the largest index in set $J$, we expand the determinant $\mathrm{det}(\mathbf{M}_\kappa)$ by its $j_1$-th column and obtain that
$
\mathrm{det}(\mathbf{M}_\kappa) = \mathrm{det}(\mathbf{M}_{\kappa,1}),
$
where the matrix $\mathbf{M}_{\kappa, 1} \in \mathbb{C}^{(N_\kappa-1) \times (N_\kappa-1)}$ is the submatrix of $\mathbf{M}_\kappa$ obtained by removing its $j_1$-th row and $j_1$-th column. We denote by $J_1$ the subset of $J$ by removing $j_1$ from it, that is, $J_1:=J\setminus\{j_1\}$, and observe that all entries in the $j$-th column of matrix $\mathbf{M}_{\kappa, 1}$ are zeros, but one, the $j$-th entry, which is $1$. We denote by $\mathbf{M}_{\kappa, 2}$ the submatrix of  $\mathbf{M}_{\kappa, 1}$ obtained by removing its $j_2$-th row and $j_2$-th column, where $j_2$ denotes the largest index in set $J_2$. Clearly, we have that $\mathrm{det}(\mathbf{M}_{\kappa, 2}) = \mathrm{det}(\mathbf{M}_{\kappa, 1}).$
Since $\#J=N_\kappa-p_\kappa-1$, we can repeat this process $N_\kappa-p_\kappa-1$ times until obtaining $\mathbf{M}_{\kappa, N_\kappa-p_{_\kappa}-1} \in \mathbb{C}^{(p_{_\kappa}+1) \times (p_{_\kappa}+1)}$, 
yielding
\begin{equation}\label{submatrix}
    \mathrm{det}(\mathbf{M}_\kappa)=\mathrm{det}(\mathbf{M}_{\kappa, j}), \ \ j=1, 2, \dots, N_\kappa-p_{_\kappa}-1.
\end{equation}
Noting that matrix $\mathbf{M}_{\kappa, N_\kappa-p_{_\kappa}-1} \in \mathbb{C}^{(p_{_\kappa}+1) \times (p_{_\kappa}+1)}$ is obtained by removing all the rows and columns of the indices in the set $J$,  we have that 
\begin{equation*}(\mathbf{M}_{\kappa, N_\kappa-p_{_\kappa}-1})_{(j,l)}=
\begin{cases}
1-\lambda/p_{_\kappa}, &j=l\in \{1, p_{_\kappa}+1\},\\
-\lambda \omega_\kappa^{q|j-l|}/p_{_\kappa}, &l\in \{1, p_{_\kappa}+1\},\quad j\in \mathbb{N}_{p_{_\kappa}+1}\setminus \{l\},\\
1-2\lambda/p_{_\kappa}, &j=l\in \{2,3,\dots, p_{_\kappa}\},\\
-2\lambda \omega_\kappa^{q|j-l|}/p_{_\kappa}, &l\in \{2,3,\dots, p_{_\kappa}\},\quad j\in \mathbb{N}_{p_{_\kappa}+1}\setminus \{l\}.
\end{cases}  
\end{equation*}

According to equation \eqref{submatrix}, it suffices to consider the range of $\kappa$ for det$(\mathbf{M}_{\kappa, N_\kappa-p_{_\kappa}-1}) \neq 0$. To this end, we perform elementary  transformations on matrix $\mathbf{M}_{\kappa, N_\kappa-p_{_\kappa}-1}$ to simplify it. Precisely, we multiply the first and $(p_{_\kappa}+1)$-th columns of matrix $\mathbf{M}_{\kappa, N_\kappa-p_{_\kappa}-1}$ by $-p_{_\kappa}$ and the second column to the $p_{_\kappa}$-th column by $-p_{_\kappa}/2$ to obtain a new matrix. To describe the resultant matrix, for $d\in \mathbb{N}$ with $d\geqslant 4$, we denote by $\mathbb{C}[x]^{d\times d}$ the set of $d\times d$ matrices whose entries are polynomials of the variable $x\in \mathbb{C}$, and for $\lambda\in \mathbb{C}$, we define the matrix
$$
\mathbf{A}_d(x;\lambda):=(a_{j,k}(x;\lambda))_{j,k\in \mathbb{N}_d}\in \mathbb{C}[x]^{d\times d},
$$
where
%
$$
a_{j,l}(x;\lambda):=\begin{cases}
\lambda- (d-1), &j=l\in \{1, d\}, \\
\lambda-\frac{d-1}{2}, & j=l\in\{2,3,\dots, d-1\},\\
\lambda x^{|j-l|}, & \mathrm{otherwise.}
\end{cases}
$$
It can be verified that the matrix that results from matrix $\mathbf{M}_{\kappa, N_\kappa-p_{_\kappa}-1}$ by the elementary transformations described above is $\mathbf{A}_{p_{_\kappa}+1}(\omega_\kappa^q; \lambda)$. Noticing that $\omega_\kappa^q = e^{\frac{2\kappa}{p_{_\kappa}}i}$ by the definition \eqref{def_omega_kappa} of $\omega_\kappa$, we conclude that the matrix $\mathbf{M}_\kappa$ is invertible if and only if 
$
\mathrm{det}\left(\mathbf{A}_{p_{_\kappa}+1}\left(e^{\frac{2\kappa}{p_{_\kappa}}i};\lambda\right)\right) \neq 0. 
$
We now introduce the set
\begin{equation}
    S(\lambda):=\left\{\kappa \geqslant 1:  {\rm det}\left(\mathbf{A}_{p_{_\kappa}+1}\left(e^{\frac{2 \kappa}{p_{_\kappa}}i}; \lambda\right)\right)\neq 0\right\} \label{def_s}
\end{equation}
and summarize the discussion above in the following lemma. 

\begin{lemma} \label{lemma_invertible}
Matrix $\mathbf{M}_\kappa\in \mathbb{C}^{N_\kappa\times N_\kappa}$ is invertible if and only if $\kappa\in S(\lambda)$. 
\end{lemma}


Lemma \ref{lemma_invertible} indicates that matrix $\mathbf{M}_\kappa^{-1}$ exists if and only if $\kappa\in S(\lambda)$. Hence, the norm $\|\mathbf{M}_\kappa^{-1}\|_2$ is defined only for $\kappa\in S(\lambda)$. Is this too restricted? The next proposition reveals that the complementary set $[1,+\infty)\setminus S(\lambda)$ is at most countable. This in turn implies the set $S(\lambda)$ is equal to $[1, +\infty)$ almost everywhere, and thus, it is not too restricted.

\begin{proposition} \label{prop_most_countable}
  For any $\lambda\in \mathbb{C}$, the set $[1,+\infty)\setminus S(\lambda)$ is at most countable.
\end{proposition}
\begin{proof}
We prove this proposition by decomposing the set $[1,+\infty)\setminus S(\lambda)$. By the definition \eqref{def_s}  of $S(\lambda)$, we know that
\begin{equation}
    [1,+\infty)\setminus S(\lambda)=\left\{\kappa \geqslant 1:  {\rm det}\left(\mathbf{A}_{p_{_\kappa}+1}\left(e^{\frac{2 \kappa}{p_{_\kappa}}i}; \lambda\right)\right)= 0\right\}. \label{def_s_inv}
\end{equation}
Equation \eqref{def_s_inv} reveals that set $[1,+\infty)\setminus S(\lambda)$ is determined by the zeros of polynomial ${\rm det}(\mathbf{A}_{p_{_\kappa}+1}\left(x; \lambda\right))$, $x\in \mathbb{C}$. We decompose set $[1,+\infty)\setminus S(\lambda)$ according to the order, related to $p_{\kappa}+1$, of the polynomial. Since $p_{_\kappa}+1\geqslant4$ by the definition \eqref{def_p_kappa} of $p_{_\kappa}$, for $d\geqslant 4$, we define
\begin{equation}
    S_d(\lambda):= \{x\in \mathbb{C}: \mathrm{det}(\mathbf{A}_{d}(x;\lambda))=0\}.
    \label{def_s_d}
\end{equation}
Thus, we may decompose  set $[1,+\infty)\setminus S(\lambda)$ as
\begin{equation}
    [1,+\infty)\setminus S(\lambda)=\bigcup_{p=3}^{\infty}\left\{\kappa\geqslant 1: e^{\frac{2\kappa}{p}i}\in S_{p+1}(\lambda), p=p_{_\kappa}\right\}.\label{trans_another_S_lambda}
\end{equation}
Meanwhile, for any $p\in\{3,4,\dots\}$, with a direct computation,  $\kappa\geqslant 1$ satisfies $p=p_{_\kappa}=\lceil \gamma \kappa^\beta\rceil$ if and only if
$$
\kappa\in I_p:= \left(\left((p-1)/\gamma\right)^{1/\beta}, \left({p}/{\gamma}\right)^{1/\beta}\right] \cap[1,+\infty).
$$
Hence, equation \eqref{trans_another_S_lambda} may be written as
\begin{equation*}
    [1,+\infty)\setminus S(\lambda)=\bigcup_{p=3}^{\infty}\left\{\kappa\in I_p: e^{\frac{2\kappa}{p}i}\in S_{p+1}(\lambda)\right\},
\end{equation*}
that is, 
\begin{equation}
    [1,+\infty)\setminus S(\lambda)=\bigcup_{p=3}^{\infty}\bigcup_{z\in S_{p+1}(\lambda)}\left\{\kappa\in I_p : e^{\frac{2\kappa}{p}i}=z \right\}.\label{trans3_another_S_lambda}
\end{equation}

According to decomposition \eqref{trans3_another_S_lambda}, it suffices to prove that
\begin{equation}
    \#(S_{p+1}(\lambda))<+\infty, \quad p\in \{3,4,\dots\} \label{porp_set_to_prove1}
\end{equation}
and
\begin{equation}
    \#\left(\left\{\kappa\in I_p : e^{\frac{2\kappa}{p}i}=z \right\}\right)<+\infty,  \quad z\in S_{p+1}(\lambda)\quad p\in \{3,4,\dots\}. \label{porp_set_to_prove2}
\end{equation}
We first establish \eqref{porp_set_to_prove1}. For each $p\geqslant 3, \lambda \in \mathbb{C}$, by the definition \eqref{def_s_d} of set $S_{p+1}(\lambda)$, the cardinality  of the set corresponds to the number of zeros of the polynomial $\det(\mathbf{A}_{p+1}(x;\lambda))$, $x\in \mathbb{C}$, and this is certainly a finite number.
As for \eqref{porp_set_to_prove2}, for each $p\in \{3,4,\dots\}$, we observe that $e^{\frac{2\kappa}{p}i}$, $\kappa\geqslant 1$, as a function of $\kappa$ has a   period $p\pi$. Therefore, for each $p\in \{3,4,\dots\}$ and $z\in S_{p+1}(\lambda)$, the number of solutions of equation $e^{\frac{2i\kappa}{p}}=z$ in the domain $I_p$
is no larger than $((p/\gamma)^{1/\beta}-((p-1)/\gamma)^{1/\beta})/(p\pi)$, which is clearly finite. This ensures that inequality \eqref{porp_set_to_prove2} holds.
\end{proof}

Proposition \ref{prop_most_countable} indicates that restriction $\kappa\in S(\lambda)$ is reasonable.  We now return to the error estimate. 

\begin{lemma} \label{lem_inverse}
If $\kappa\in S(\lambda)$, then for any $h\in C(I)$, 
$$
\|h\|_{N_\kappa}\leqslant \|\mathbf{M}_\kappa^{-1}\|_2\|(\mathcal{I}-\lambda\mathcal{K}_{p_{_\kappa}})h\|_{N_\kappa}.
$$
\end{lemma}
\begin{proof} 
Since $\kappa\in S(\lambda)$, Lemma \ref{lemma_invertible} ensures that the inverse $\mathbf{M}^{-1}_\kappa$ exists.
Hence, by the property of the $l_2$ norm, we derive for any $\mathbf{u}\in \mathbb{C}^{N_\kappa}$ that
\begin{equation}
    \|\mathbf{M}_\kappa^{-1}\mathbf{u}\|_{l_2}\leqslant \|\mathbf{M}_\kappa^{-1}\|_2\|\mathbf{u}\|_{l_2}. \label{prop_C_kappa:equ3}
\end{equation}
Now, for any $h\in C(I)$, we define $\mathbf{v}_h:=[h(x_j), j\in \mathbb{N}_{\kappa}]^T \in \mathbb{C}^{N_\kappa}$. Applying inequality \eqref{prop_C_kappa:equ3} to $\mathbf{u}:=\mathbf{M}_\kappa \mathbf{v}_h$ yields that
$$
\|\mathbf{v}_h\|_{l_2}\leqslant \|\mathbf{M}_\kappa^{-1}\|_2\|\mathbf{M}_\kappa \mathbf{v}_h\|_{l_2}.    
$$
This together with the equations
$\|h\|_{N_{\kappa}} = \|\mathbf{v}_h\|_{l_2}/\sqrt{N_\kappa}$ and  $\|(\mathcal{I}-\lambda\mathcal{K}_{p_{_\kappa}})h\|_{N_\kappa}=\left.\left\|\mathbf{M}_\kappa \mathbf{v}_h\right\|_{l_2}\right/\sqrt{N_\kappa}$
obtained from relation \eqref{relation_N_2} and \eqref{def_M_kappa}, respectively, we conclude the desired estimate.
\end{proof}

We next integrate Lemmas \ref{lem_err_before_p} and \ref{lem_inverse} to yield an estimate of the error $\|y-\tilde{y}^*\|_{N_\kappa}$.

\begin{theorem}\label{lemma_y-yi}
If $\kappa\in S(\lambda)$, then
\begin{equation}
\|y-\tilde{y}^*\|_{N_\kappa} \leqslant \|\mathbf{M}_\kappa^{-1}\|_2\left( \|\tilde{e}^*\|_{N_\kappa}+|\lambda|\|(\mathcal{K}-\mathcal{K}_{p_{_\kappa}})y\|_\infty\right). \label{lemma_y-yi:target}
\end{equation}
\end{theorem}
\begin{proof}
Applying Lemma \ref{lem_inverse} with $h:=y-\tilde{y}^*$ yields
$$
\|y-\tilde{y}^*\|_{N_\kappa}\leqslant \|\mathbf{M}_\kappa^{-1}\|_2\|(\mathcal{I}-\lambda\mathcal{K}_{p_{_\kappa}})(y-\tilde{y}^*)\|_{N_\kappa}. 
$$
This together with inequality \eqref{lem_err_before_p:target} in Lemma \ref{lem_err_before_p} leads to the desired estimate of this lemma.
\end{proof}

\begin{figure*}[!ht]
    \centering
    \subfigure[$(\lambda, \gamma, \beta, q)=(1, 6, 1, 1)$.]{
    	\begin{minipage}{0.45\textwidth}
           \includegraphics[width=1\textwidth]{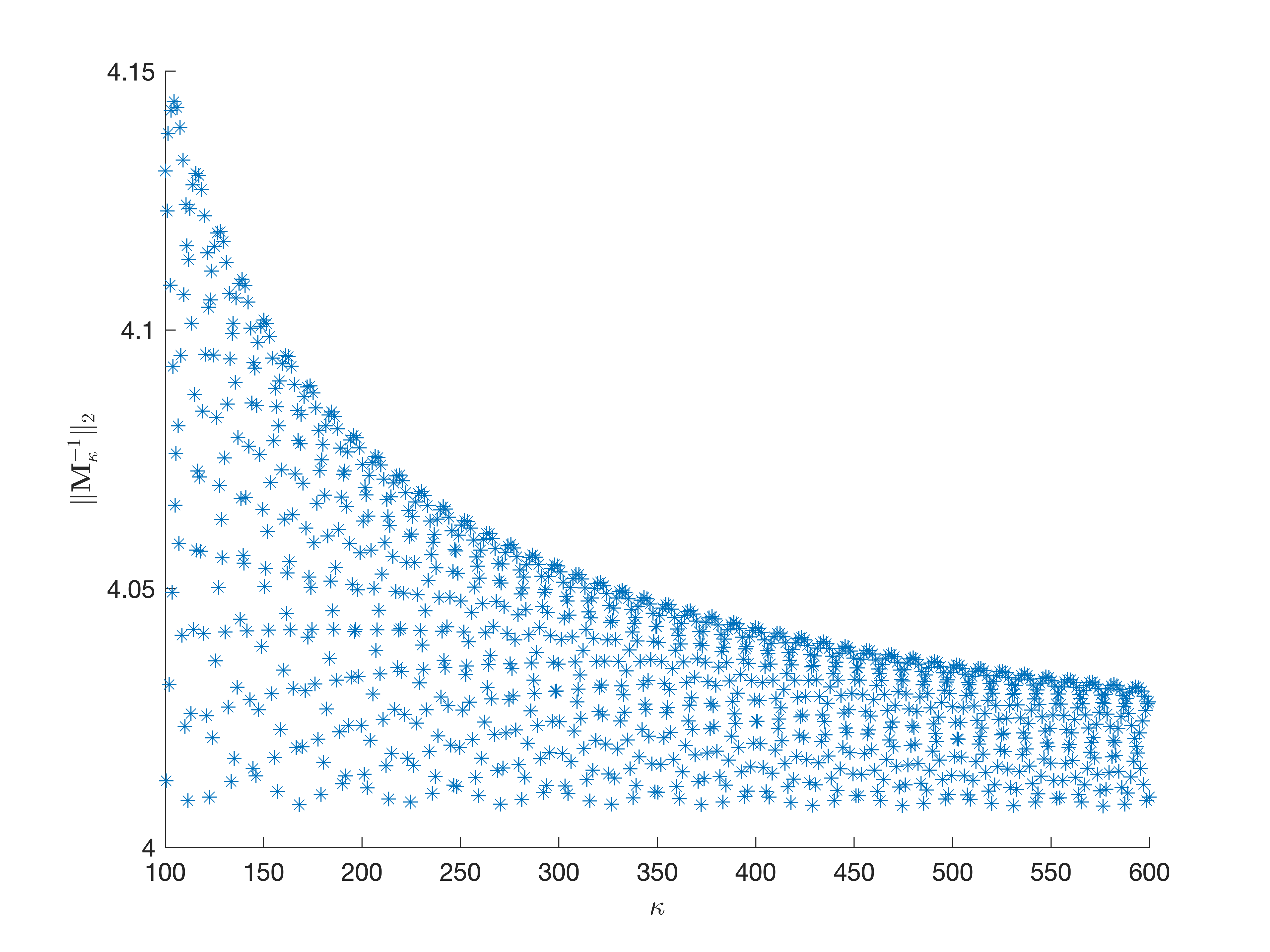}
           \label{lambda_1}
           \end{minipage}}
      \hfil
    \subfigure[$(\lambda, \gamma, \beta, q)=(2, 6, 1, 1)$.]{
    	\begin{minipage}{0.45\textwidth}
           \includegraphics[width=1\textwidth]{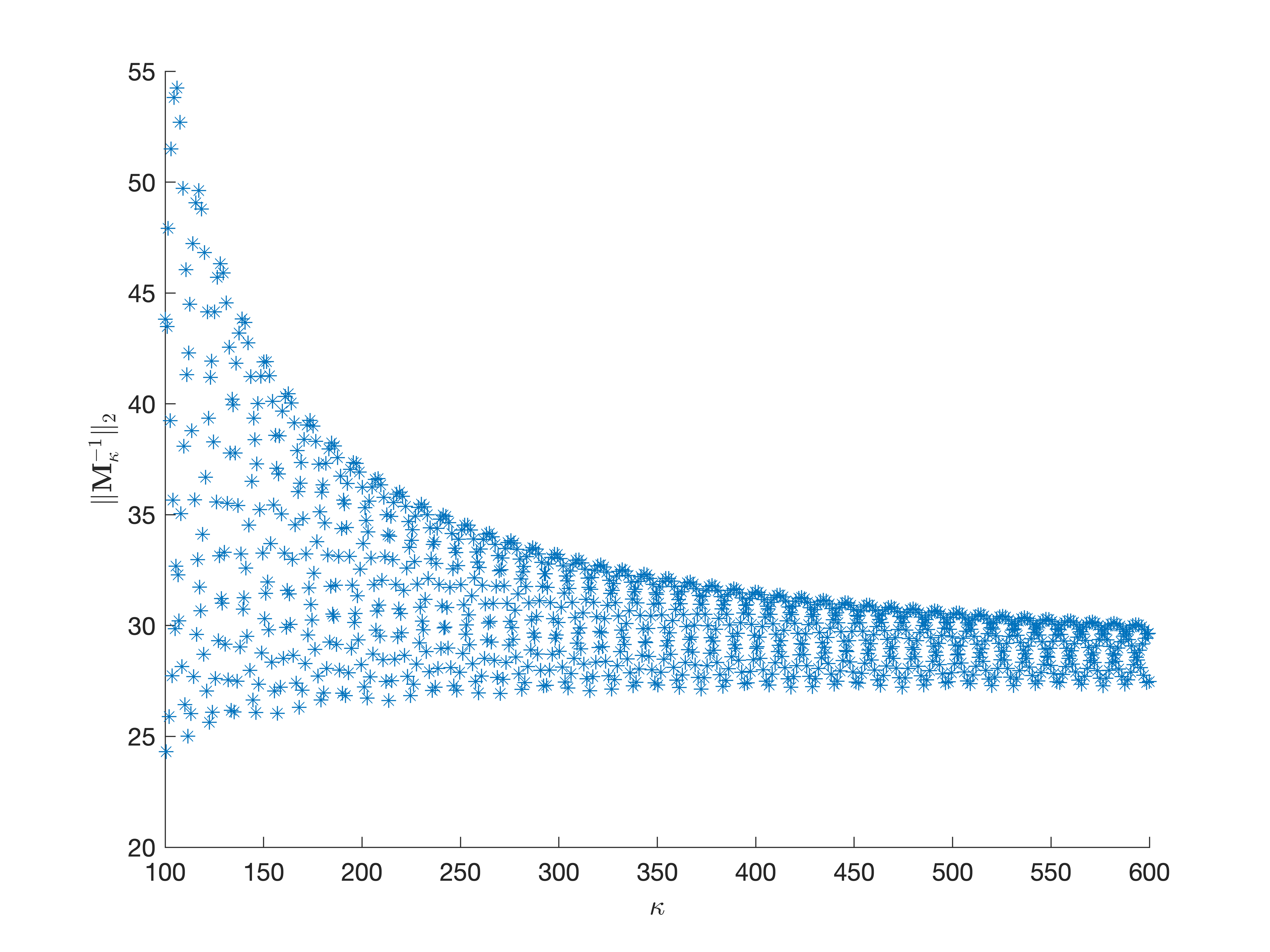}
           \label{lambda_2}
    	\end{minipage}
    }
   \qquad
   \subfigure[$(\lambda, \gamma, \beta, q)=(3, 6, 1, 1)$.]{
    	\begin{minipage}{0.45\textwidth}
           \includegraphics[width=1\textwidth]{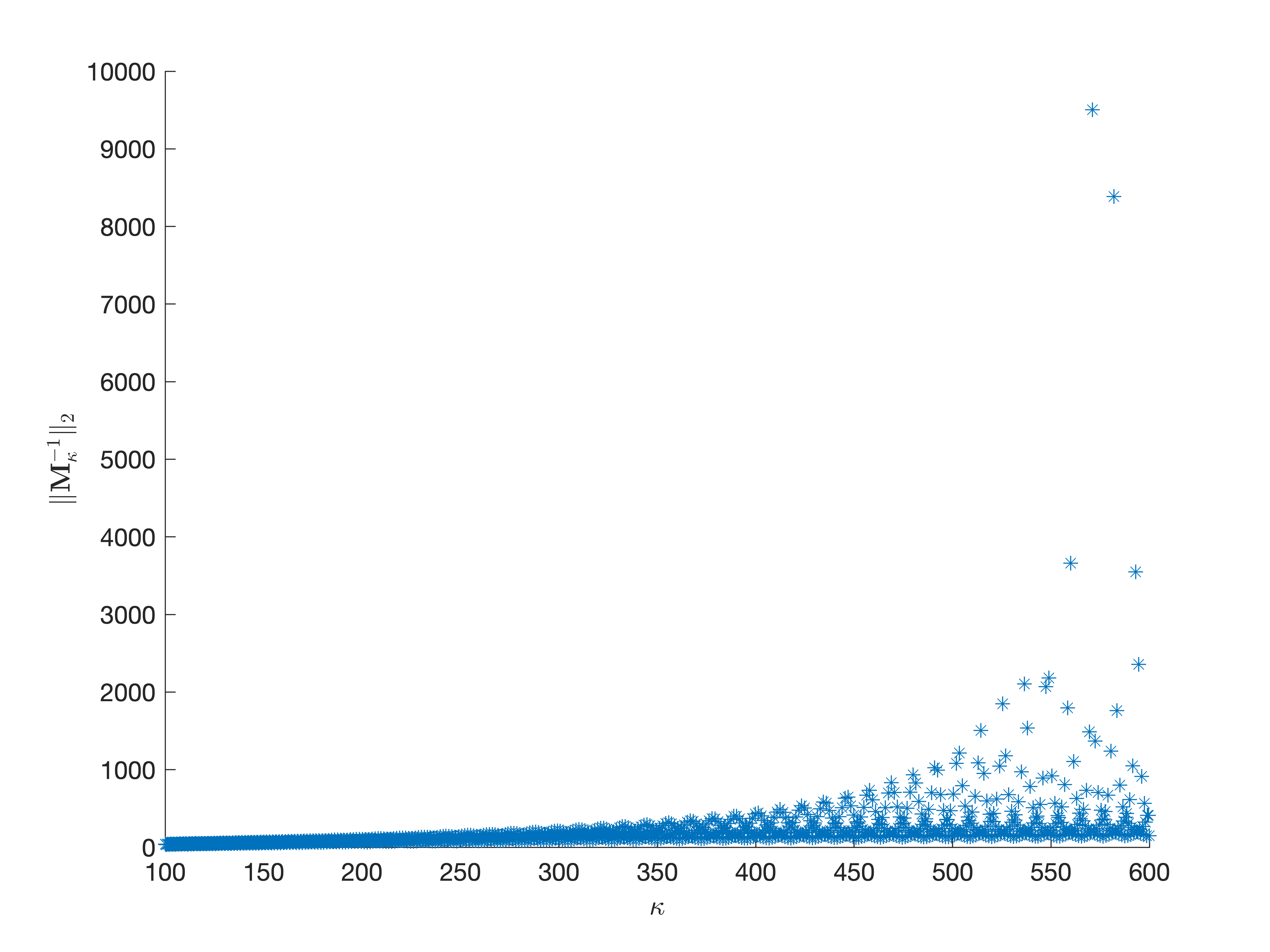}
           \label{lambda_3}
    	\end{minipage}
      }
      \hfil
    \subfigure[$(\lambda, \gamma, \beta, q)=(4, 6, 1, 1)$.]{
    	\begin{minipage}{0.45\textwidth}
           \includegraphics[width=1\textwidth]{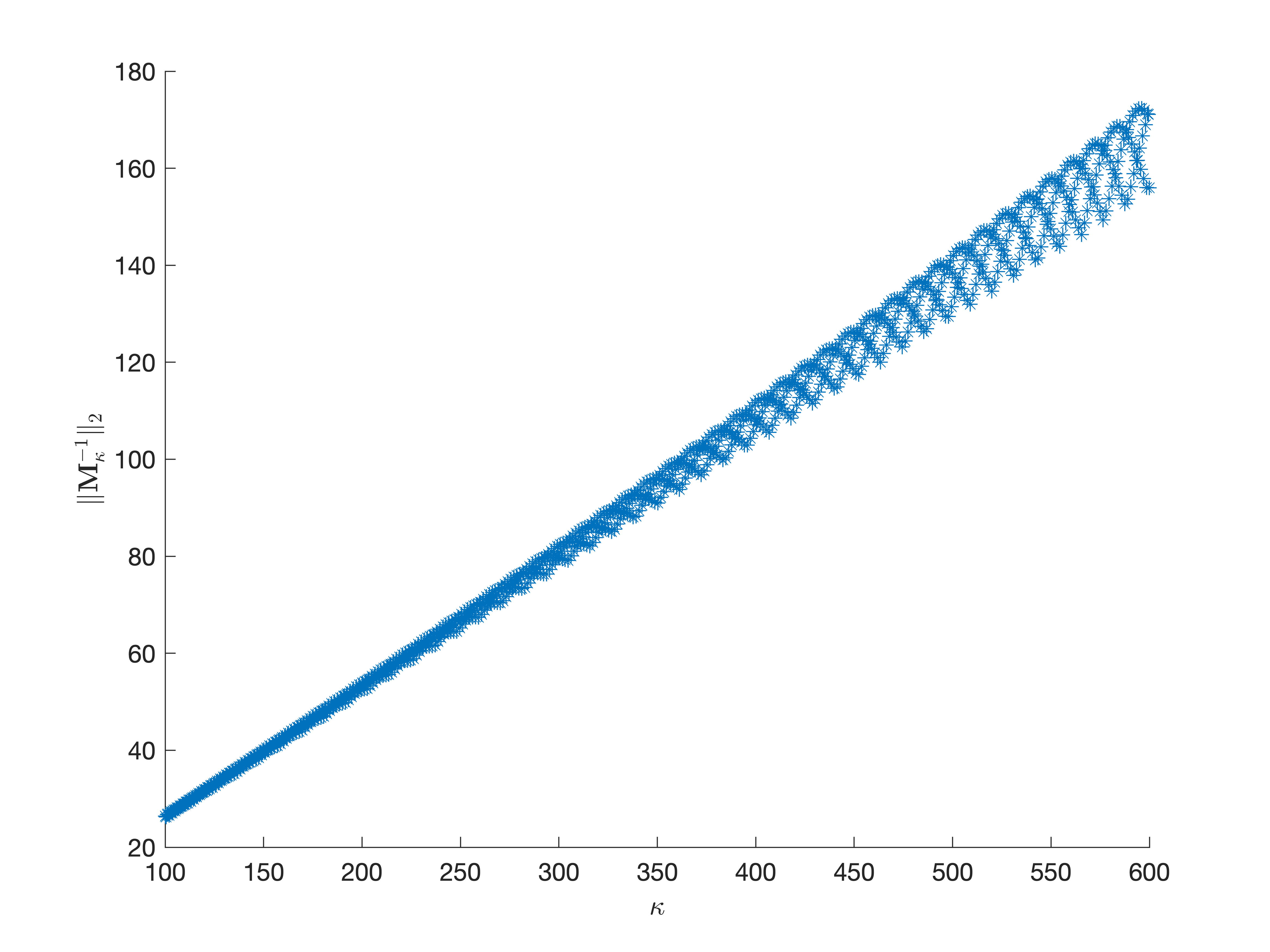}
           \label{lambda_4}
    	\end{minipage}
    }
    
    \qquad
   \subfigure[$(\lambda, \gamma, \beta, q)=(5, 6, 1, 1)$.]{
    	\begin{minipage}{0.45\textwidth}
           \includegraphics[width=1\textwidth]{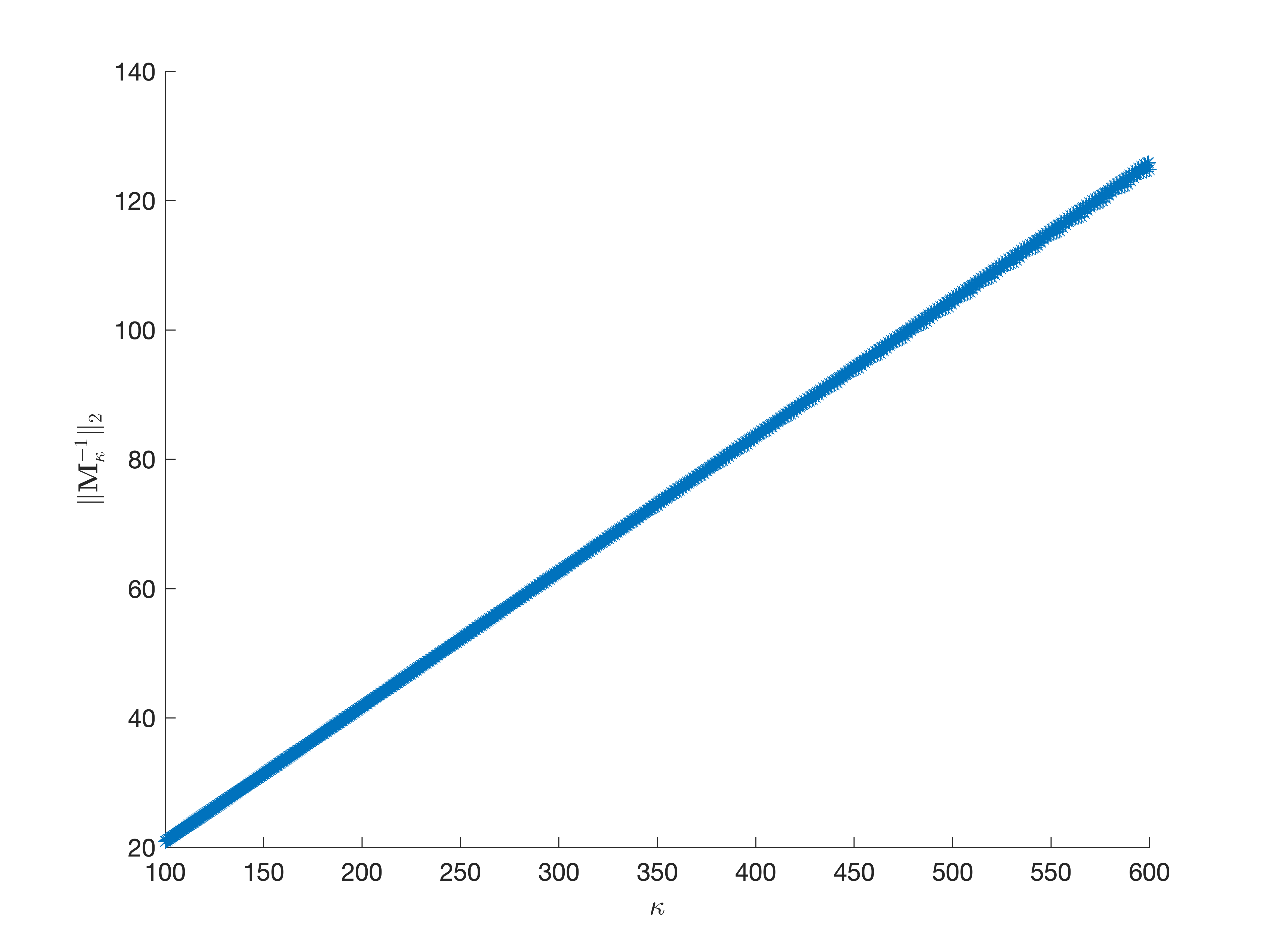}
           \label{lambda_5}
    	\end{minipage}
      }
      \hfil
    \subfigure[$(\lambda, \gamma, \beta, q)=(6, 6, 1, 1)$.]{
    	\begin{minipage}{0.45\textwidth}
           \includegraphics[width=1\textwidth]{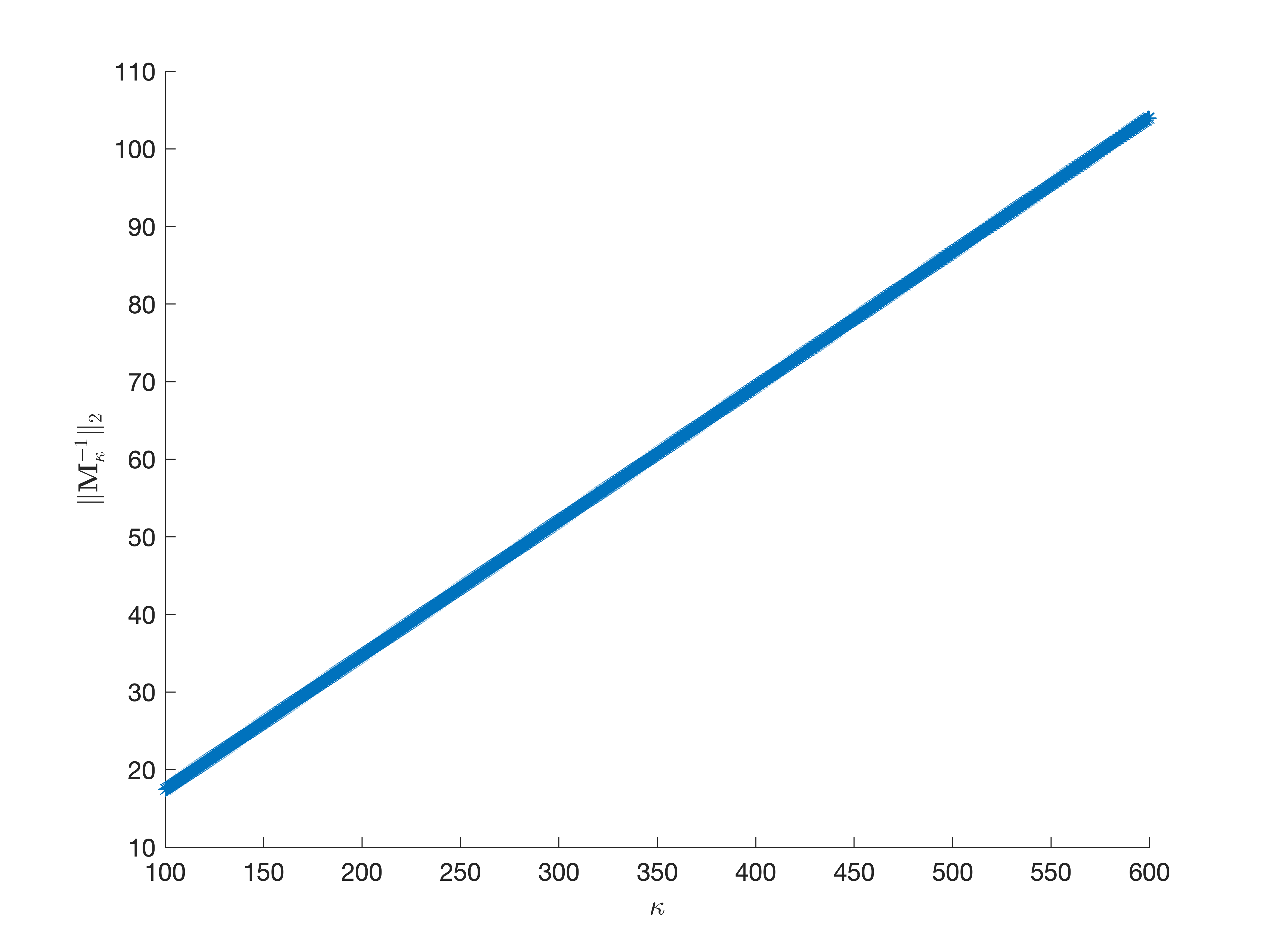}
           \label{lambda_6}
    	\end{minipage}
    }
    \end{figure*}
    \begin{figure*}
        \centering
   \subfigure[$(\lambda, \gamma, \beta, q)=(7, 6, 1, 1)$.]{
    	\begin{minipage}{0.45\textwidth}
           \includegraphics[width=1\textwidth]{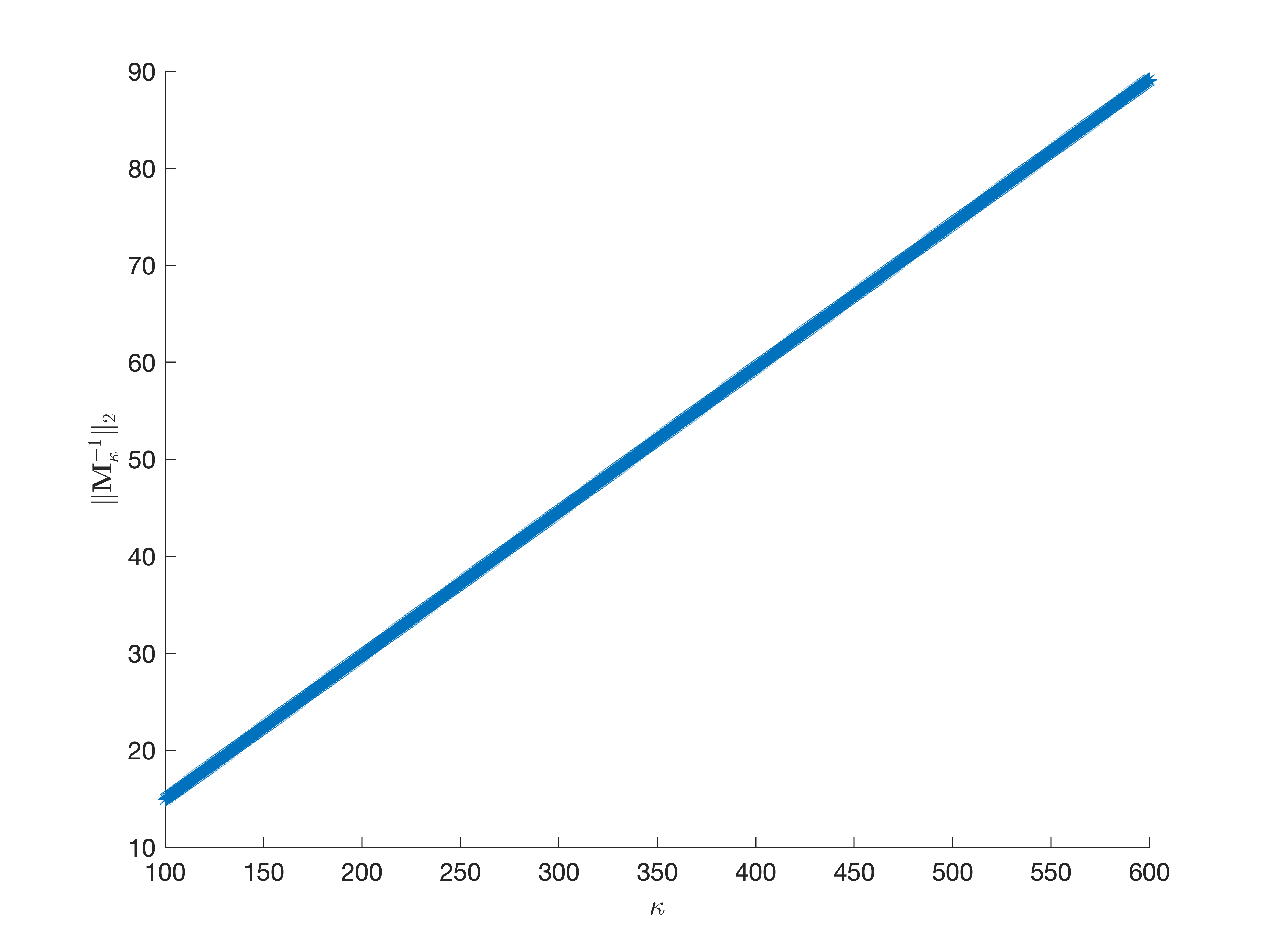}
           \label{lambda_7}
    	\end{minipage}
      }
      \hfil
    \subfigure[$(\lambda, \gamma, \beta, q)=(8, 6, 1, 1)$.]{
    	\begin{minipage}{0.45\textwidth}
           \includegraphics[width=1\textwidth]{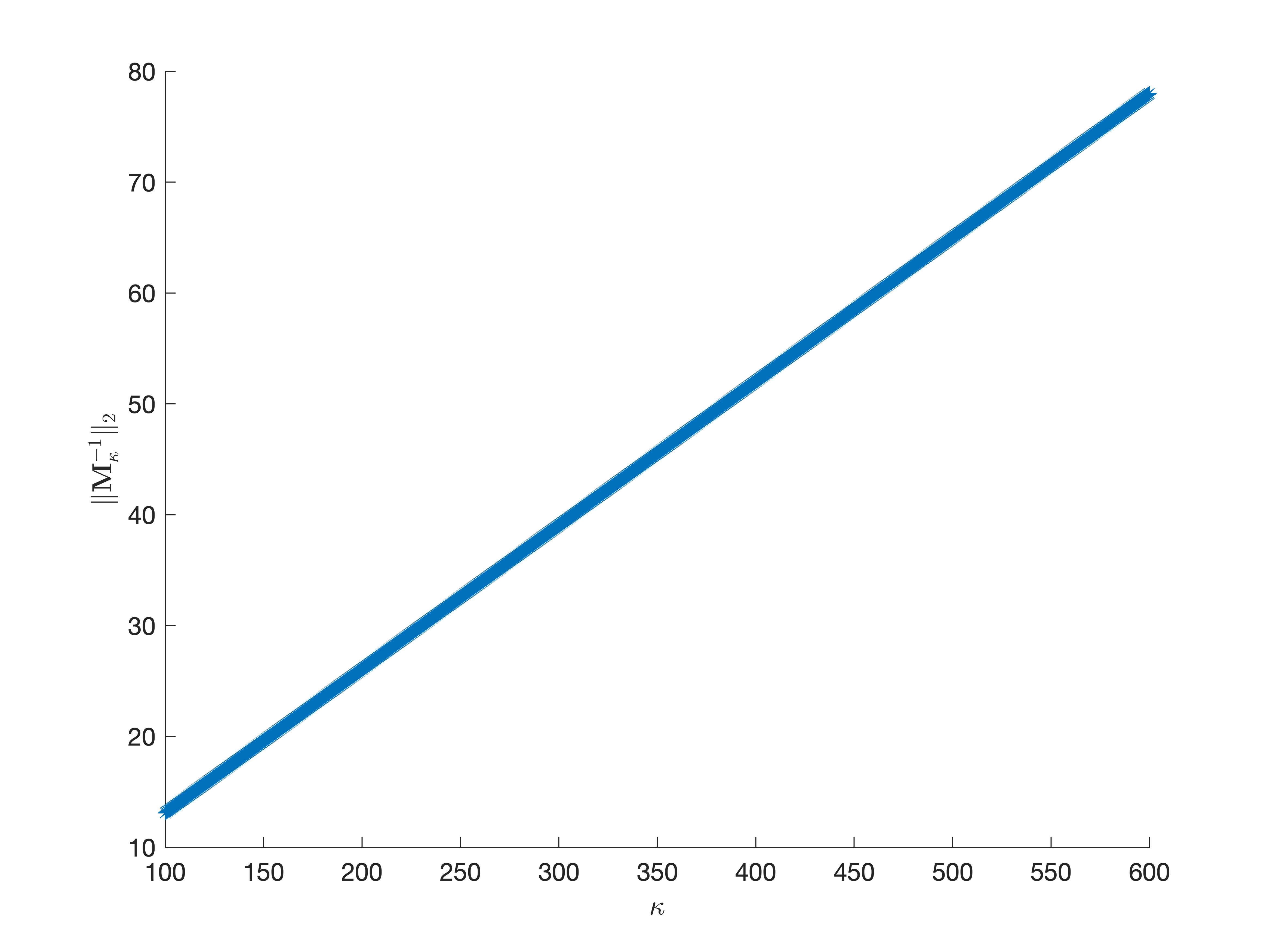}
           \label{lambda_8}
    	\end{minipage}
    }
    \qquad
   \subfigure[$(\lambda, \gamma, \beta, q)=(9, 6, 1, 1)$.]{
    	\begin{minipage}{0.45\textwidth}
           \includegraphics[width=1\textwidth]{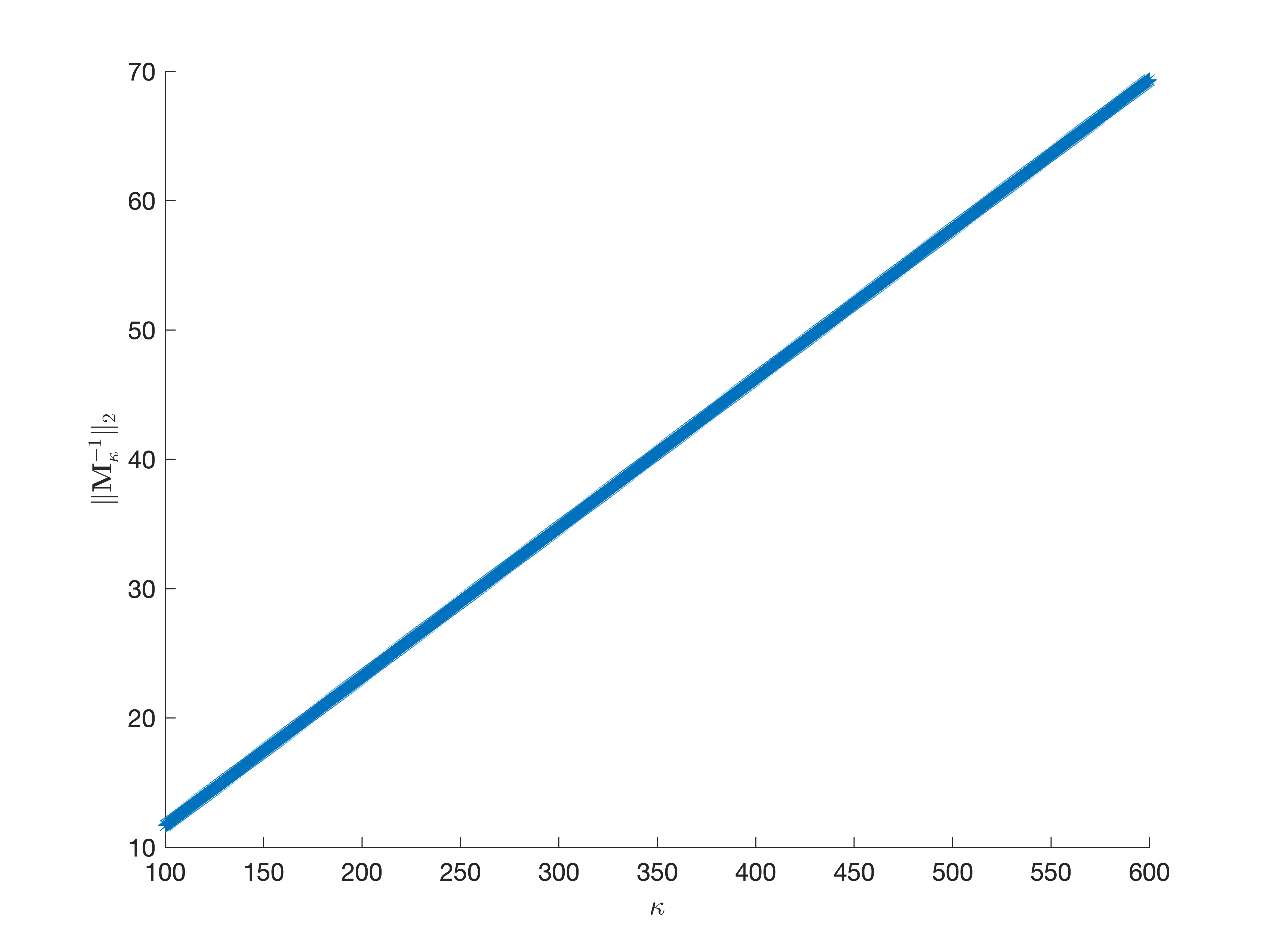}
           \label{lambda_9}
    	\end{minipage}
      }
      \hfil
    \subfigure[$(\lambda, \gamma, \beta, q)=(10, 6, 1, 1)$.]{
    	\begin{minipage}{0.45\textwidth}
           \includegraphics[width=1\textwidth]{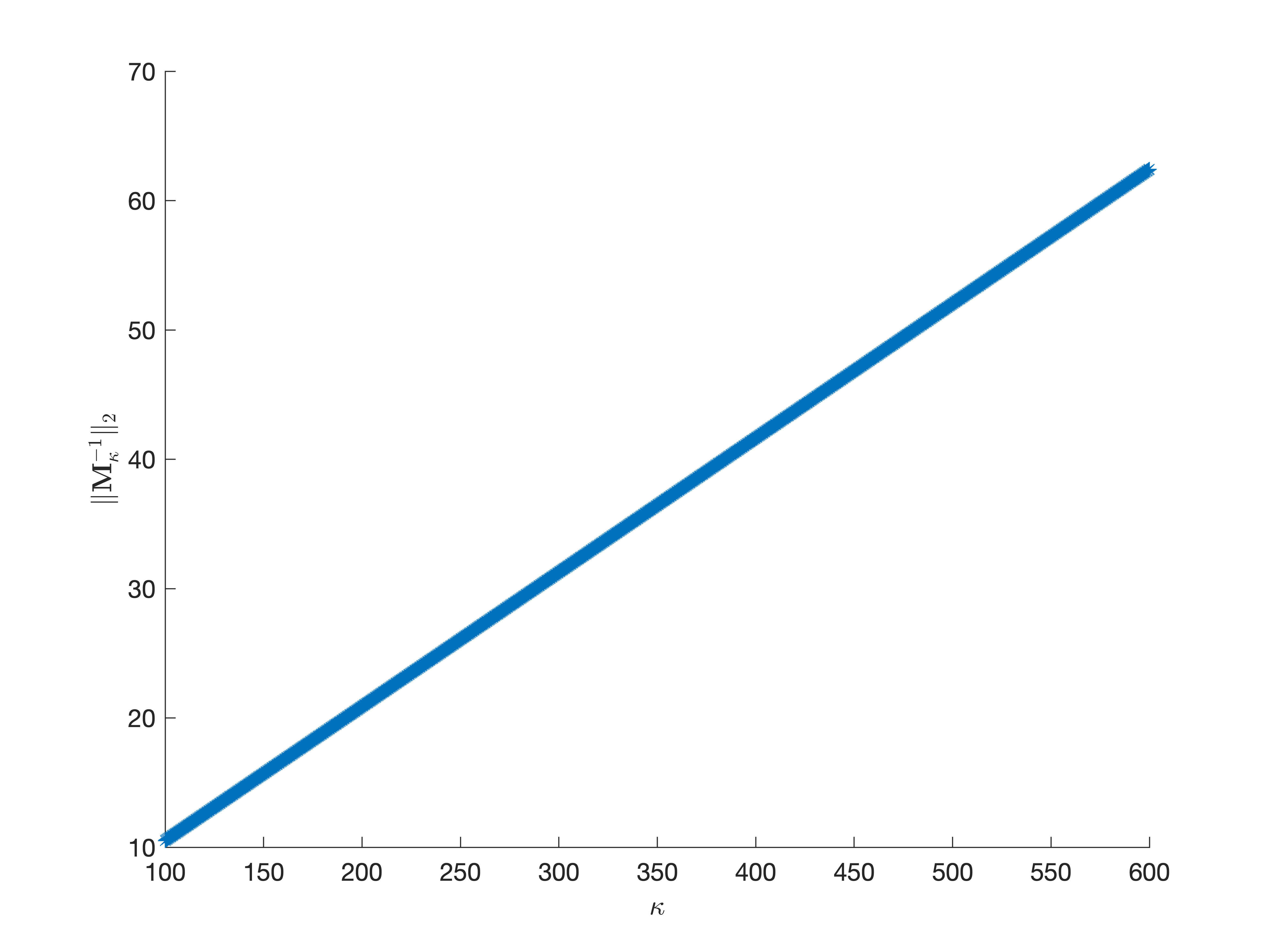}
           \label{lambda_10}
    	\end{minipage}
     }
     \qquad
   \subfigure[$(\lambda, \gamma, \beta, q)=(100, 6, 1, 1)$.]{
    	\begin{minipage}{0.45\textwidth}
           \includegraphics[width=1\textwidth]{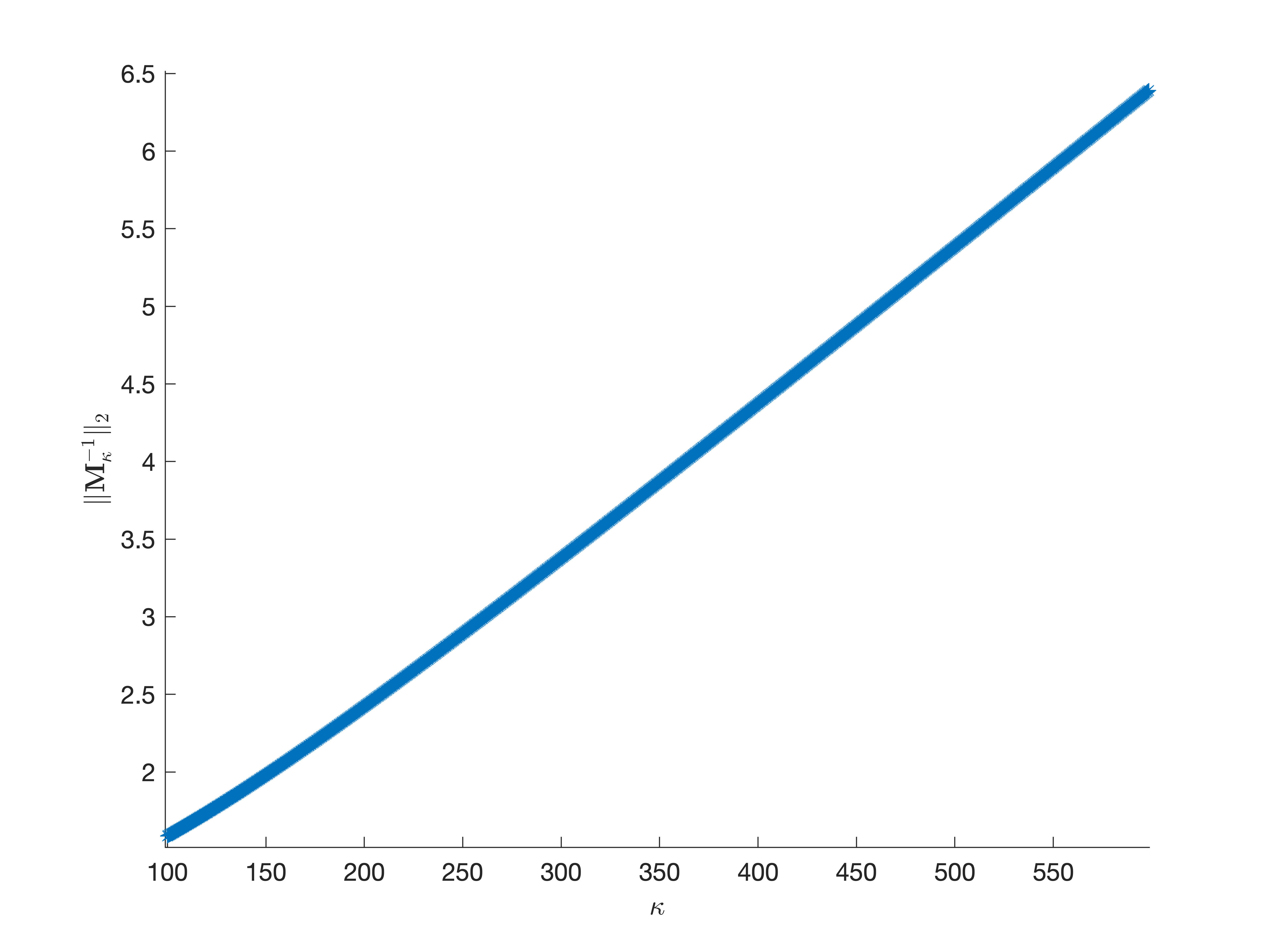}
           \label{lambda_100}
    	\end{minipage}
      }
      \hfil
    \subfigure[$(\lambda, \gamma, \beta, q)=(1000, 6, 1, 1)$.]{
    	\begin{minipage}{0.45\textwidth}
           \includegraphics[width=1\textwidth]{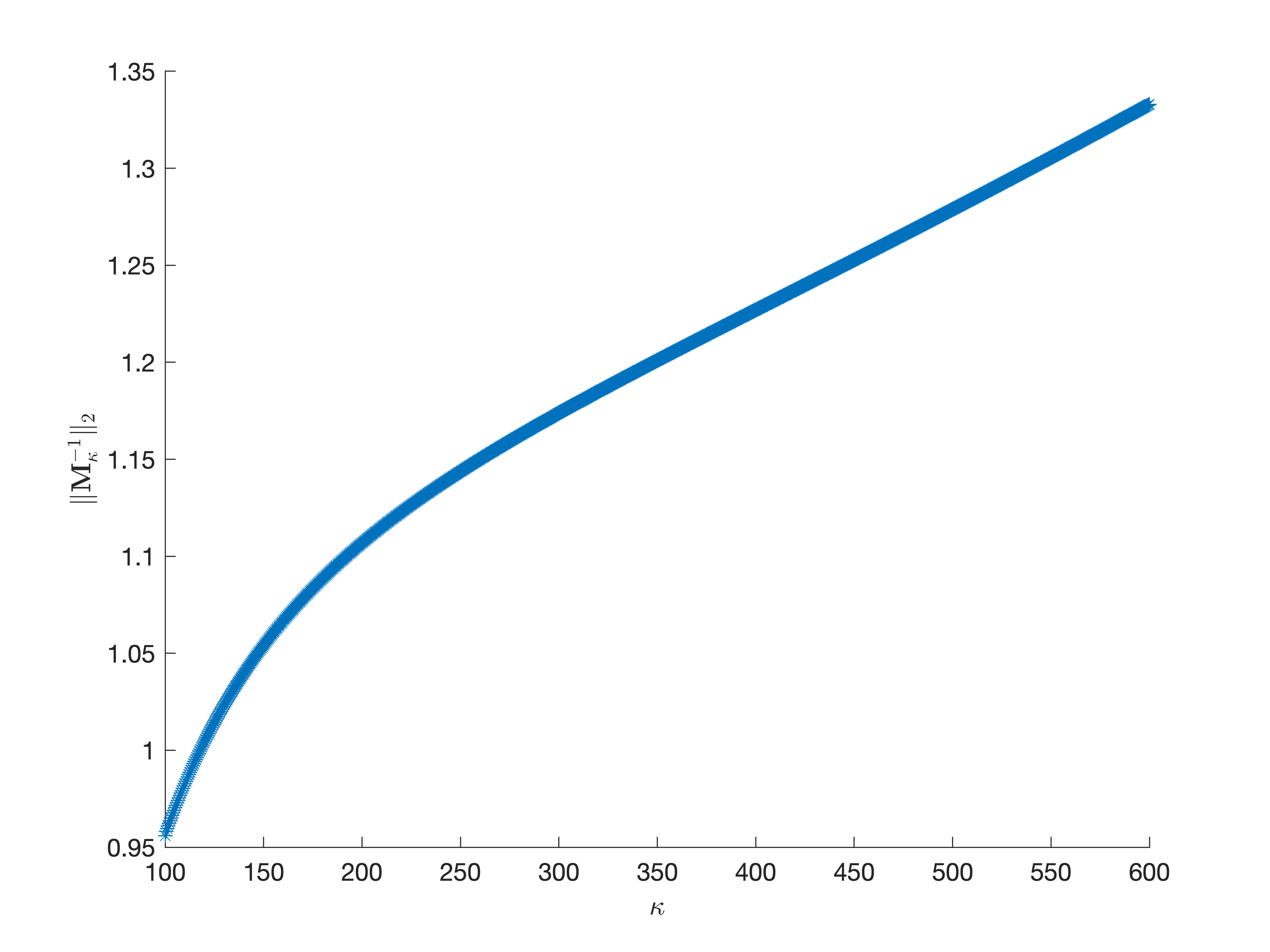}
           \label{lambda_1000}
    	\end{minipage}
    }
    \caption{Values of $\|\mathbf{M}_{\kappa}^{-1}\|_{N_\kappa}$ as a function of $\kappa$ for different $(\lambda, \gamma, \beta, q)$.}
    \label{figure:M_kappa}
\end{figure*}

The coefficient $\|\mathbf{M}_\kappa^{-1}\|_2$ of the estimate \eqref{lemma_y-yi:target} depends on $\kappa$. Figure \ref{figure:M_kappa} illustrates the dependence of  $\|\mathbf{M}_\kappa^{-1}\|_2$ on $\kappa$ for $\lambda=1, 2, \dots, 10, 100, 1000$. The figure indicates that  $\|\mathbf{M}_\kappa^{-1}\|_2$ is either bounded above by a constant when $\lambda=1,2$ or increasing ``linearly'' as $\kappa$ increases when $\lambda=3, 4, \dots, 10, 100, 1000$. In all the linear cases, the largest $\|\mathbf{M}_\kappa^{-1}\|_2$ value is 180 for $\lambda=4$ at $\kappa=600$.
It is desirable to have  $\|\mathbf{M}_\kappa^{-1}\|_2$  independent of $\kappa$, which occurs for $\lambda=1,2,3$ in Figure \ref{figure:M_kappa}.
For this purpose, we make the following hypothesis.

\begin{hypothesis} \label{hypothesis_A}
There exists a $\kappa_0\geqslant 1$ and a constant $C>0$ such that $\|\mathbf{M}_\kappa^{-1}\|_2\leqslant C$ for all $\kappa\in S(\lambda)\cap [\kappa_0, +\infty)$. 
\end{hypothesis}

Under Hypothesis \ref{hypothesis_A}, Theorem \ref{lemma_y-yi} can be strengthened.

\begin{theorem} \label{thm_y-yi}
Suppose the solution $y\in H^m_{\kappa, 0}(I)$ of the oscillatory Fredholm integral equation \eqref{fredholm_equation_operator} has the form $\eqref{solution}$ and Hypothesis \ref{hypothesis_A} holds.
If parameters are chosen to satisfy  
\begin{equation*}
    \Gamma \geqslant 0, \quad \beta\geqslant 1, \quad \gamma\geqslant \Gamma+3,\quad q\in \mathbb{N}, 
\end{equation*} 
and for any $\kappa\in S(\lambda)\cap [\kappa_0, \infty)$, let $p_{_\kappa}:=\lceil \gamma \kappa^\beta \rceil$, $N_\kappa:=p_{_\kappa}q+1$, then
\begin{equation*}
\|y-\tilde{y}^*\|_{N_\kappa} \leqslant C\left( \|\tilde{e}^*\|_{N_\kappa}+|\lambda|\|(\mathcal{K}-\mathcal{K}_{p_{_\kappa}})y\|_\infty\right). 
\end{equation*}
\end{theorem}
\begin{proof}
This theorem follows directly from Theorem \ref{lemma_y-yi} and Hypothesis \ref{hypothesis_A}.
\end{proof}

To close this section, we show that for $\lambda$ satisfying $|\lambda|\in (0, \frac{1}{2})$ , how parameters $\beta, \gamma, q$ may be chosen to ensure the validity of Hypothesis \ref{hypothesis_A}. 


\begin{lemma} \label{lemma_inverse}
    For $\lambda$ satisfying $|\lambda|\in (0,\frac{1}{2})$,  if the parameters are chosen to satisfying 
    \begin{equation}
   \Gamma\geqslant 0,  \quad    \beta\geqslant 1,\quad  \gamma > \max\left\{\Gamma+3, \frac{4|\lambda|^2}{1-4|\lambda|^2}\right\}, \quad  q\in \left[ 1, \frac{1}{4|\lambda|^2}-\frac{1}{\lceil \gamma\rceil}\right)\cap \mathbb{N}, \label{rule_beta_gamma_q}
    \end{equation} 
$p_{_\kappa}:=\lceil \gamma \kappa^\beta\rceil$ and $N_\kappa:=p_{_\kappa} q+1$, then for any $\kappa \geqslant 1$, the matrix $\mathbf{M}_\kappa \in \mathbb{C}^{N_\kappa\times N_\kappa}$ as defined in equation \eqref{def_M_kappa} is invertible and
\begin{equation}
        \left\| \mathbf{M}_\kappa^{-1} \right\|_2\leqslant \frac{1}{1-\eta}, \label{lemma_inverse_target}
\end{equation}
where
     \begin{equation}
    \eta:= 2|\lambda|\sqrt{q+\frac{1}{\lceil \gamma \rceil}} \label{def_eta}
   \end{equation}
satisfies the condition $0<\eta<1$. 
\end{lemma}
\begin{proof}
First, we show the existence of parameters $\Gamma$, $\beta$, $\gamma$, and $q$ that satisfy condition \eqref{rule_beta_gamma_q}.  Clearly, we can find $\Gamma, \beta, \gamma$ according to \eqref{rule_beta_gamma_q}. By the choice of $\gamma$, we know  that $\lceil \gamma\rceil>\frac{4|\lambda|^2}{1-4|\lambda|^2}$, which implies that
$\frac{1}{4|\lambda|^2}-\frac{1}{\lceil \gamma\rceil} > 1$.
Thus, we can find a $q$ as in \eqref{rule_beta_gamma_q}.

Next we estimate $\left\| \lambda\mathbf{B}_\kappa/p_{_\kappa}\right\|_2$.
Note that $\left\| \mathbf{B}_\kappa\right\|^2_2\leqslant \left\|\mathbf{B}_\kappa\right\|_1 \left\|\mathbf{B}_\kappa\right\|_\infty$. This bound together with  
$ \left\|\mathbf{B}_\kappa\right\|_1 = 2(qp_{_\kappa}+1)$ and $\left\|\mathbf{B}_\kappa\right\|_\infty =2p_{_\kappa}$ yields that
$\left\|\mathbf{B}_\kappa\right\|_2
    \leqslant 2\sqrt{qp_{_\kappa}^2+p_{_\kappa}}$, which implies
\begin{equation}\label{A-Bound-onB_kappa}
        \left\|\frac{\lambda }{p_{_\kappa}}\mathbf{B}_\kappa\right\|_2
    \leqslant 2|\lambda|\sqrt{q+\frac{1}{p_{_\kappa}}}. 
\end{equation}
Meanwhile, by the definition \eqref{def_p_kappa} of $p_{_\kappa}$, we know that $p_{_\kappa}=\lceil \gamma \kappa^\beta\rceil \geqslant \lceil \gamma\rceil$ for all $\kappa\geqslant 1$. Using it in inequality \eqref{A-Bound-onB_kappa}, we observe for all $\kappa\geqslant 1$ that
\begin{equation}
        \left\|\frac{\lambda }{p_{_\kappa}}\mathbf{B}_\kappa\right\|_2\leqslant 2|\lambda|\sqrt{q+\frac{1}{\lceil \gamma \rceil}}=\eta,\label{lemma_inverse_equ3}
\end{equation}
whose right-hand side is independent of $\kappa$. 

Furthermore, by the choice \eqref{rule_beta_gamma_q} of $q$, we notice that
$q< \frac{1}{4|\lambda|^2}-\frac{1}{\lceil \gamma\rceil}$. Substituting it into the definition \eqref{def_eta} of $\eta$ yields that $0<\eta<1$. This estimation together with inequality \eqref{lemma_inverse_equ3}, we conclude the matrix $\mathbf{M}_\kappa=\mathbf{I}_\kappa-\frac{\lambda }{p_{_\kappa}}\mathbf{B}_\kappa$ is invertible and the inequality \eqref{lemma_inverse_target} holds.
\end{proof}

Lemma \ref{lemma_inverse} conveys that for $\lambda$ with $|\lambda|\in (0,\frac{1}{2})$, the parameters $\Gamma$, $\beta$, $\gamma$, and $q$ can be chosen according to the rule \eqref{rule_beta_gamma_q} such that Hypothesis \ref{hypothesis_A} holds with $\kappa_0:=1$ and $C:=\frac{1}{1-\eta}$, and in this case, there holds $S(\lambda)=[1,+\infty)$.
The next theorem follows from Lemma \ref{lemma_inverse} and Theorem \ref{thm_y-yi}.

\begin{theorem} \label{theorem_lambda_y-yi}
Suppose $\lambda$ satisfies $|\lambda|\in (0,\frac{1}{2})$ and the solution $y\in H^m_{\kappa, 0}(I)$ of the oscillatory Fredholm integral equation \eqref{fredholm_equation_operator}  can be written as the form \eqref{solution}. If the parameters $\Gamma$, $\beta$, $\gamma$ and $q$ are chosen to satisfy  \eqref{rule_beta_gamma_q},
and $p_{_\kappa}:=\lceil \gamma \kappa^\beta \rceil$, $N_\kappa:=p_{_\kappa}q+1$, then for all $\kappa\geqslant 1$, 
\begin{equation*}
\|y-\tilde{y}^*\|_{N_\kappa} \leqslant \frac{1}{1-\eta}\left( \|\tilde{e}^*\|_{N_\kappa}+|\lambda|\|(\mathcal{K}-\mathcal{K}_{p_{_\kappa}})y\|_\infty\right), 
\end{equation*}
where $\eta\in(0,1)$ as defined in equation \eqref{def_eta} is independent of the wavenumber $\kappa$.
\end{theorem}

Combining Theorem \ref{theorem_lambda_y-yi} with Proposition \ref{proposition_integral_y} leads to  the following corollary.

\begin{corollary}
If the assumptions of Theorem \ref{theorem_lambda_y-yi} hold, then for all $\kappa\geqslant 1$, 
\begin{equation*}
\|y-\tilde{y}^*\|_{N_\kappa} \leqslant \frac{1}{1-\eta}\left( \|\tilde{e}^*\|_{N_\kappa}+\frac{132\tau|\lambda|}{5\gamma \kappa^\beta}+\frac{81\tau|\lambda|(\Gamma+3)^m}{5\gamma^m \kappa^{m(\beta-1)}}\right), 
\end{equation*}
where $\eta\in(0,1)$ as defined in equation \eqref{def_eta} is independent of the wavenumber $\kappa$.
\end{corollary}

\section{Multi-Grade Learning Model}

The deep neural network model \eqref{DNN_opt}, which we will refer to as the single-grade learning model, has a computational issue: Solutions of the optimization problem \eqref{DNN_opt} is often trapped in a local minimizer or even a saddle point due to too many layers used in DNNs. As a result, the loss error $\|\tilde{e}^*\|_{N_\kappa}$ may not be as small as we expect. In particular, in solving the oscillatory equation, as we will demonstrate in the next section, the single-grade learning model suffers from the spectral bias, that is, approximate solutions catch only low frequency components of the exact solution. 
To address this issue, following \cite{Xu:2023aa} we develop a multi-grade learning model for numerical solutions of equation \eqref{fredholm_equation}. The multi-grade learning model introduced in \cite{Xu:2023aa} was
motivated by the human education process which arranges learning in grades.


We now describe the multi-grade learning model for numerical solutions of equation \eqref{fredholm_equation}. Recalling the DNN that appears in minimization problem \eqref{DNN_opt} has $n$ layers, we choose $L$ positive integers $n_l$, for $l=1,2, \dots, L$, such that 
$n=\sum_{l=1}^L n_l$. Instead of solving one minimization problem \eqref{DNN_opt} of $n$ layers, we solve $L$ intertwined minimization problems, which have $n_l$ layers, for $l=1,2,\dots, L$, respectively. 
For grade 1, we define the error function by
\begin{equation*}
    \tilde{e}_1\left(\{\mathbf{W}_{j}, \mathbf{b}_{j}\}_{j=1}^{n_1}; s\right):=
    \left(f- (\mathcal{I}-\lambda\mathcal{K}_{p_{_\kappa}})\mathcal{T}\mathcal{N}_{n_1}(\{\mathbf{W}_{j}, \mathbf{b}_{j}\}_{j=1}^{n_1};\cdot)\right)(s)\in C(I),\quad s\in I  \label{def_e1_dis}
\end{equation*} 
and find  $\{\mathbf{W}_{1,j}^*, \mathbf{b}_{1,j}^*\}_{j=1}^{n_1}$ by solving the optimization problem
\begin{equation*}
 \min \left\{\left\|\tilde{e}_1\left(\{\mathbf{W}_{j}, \mathbf{b}_{j}\}_{j=1}^{n_1}; \cdot \right)\right\|_{N_\kappa}^2:  \mathbf{W}_{j}\in \mathbb{R}^{m_{1,j}\times m_{1,j-1}}, \mathbf{b}_{j}\in \mathbb{R}^{m_{1,j}}, j\in \mathbb{N}_{n_1}\right\}, \label{multi_opt_problem_dis}
 \end{equation*}
with $m_{1,0}=1$ and $m_{1, n_1}=2$. 
Once the optimal parameters $\{\mathbf{W}_{1,j}^*, \mathbf{b}_{1,j}^*\}_{j=1}^{n_1}$ are learned,  we obtain the feature of grade 1 as
$$\mathbf{g}_1(s):= \mathcal{F}_{n_1-1}\left(\{\mathbf{W}_{1,j}^*, \mathbf{b}_{1, j}^*\}_{j=1}^{n_1-1};s\right),\quad  s\in I, $$
where $\mathcal{F}_{n_1-1}$ is defined as in \eqref{dnn_feature}, and the approximate solution of grade 1 is
$$
\mathbf{f}_1(s):=\mathbf{W}_{1,n_1}^*\mathbf{g}_1(s)+\mathbf{b}_{1,n_1}^*, \quad s\in I
$$
with an error defined by 
\begin{equation*}
\tilde{e}_1^*(s) := \tilde{e}_1\left(\{\mathbf{W}^*_{1,j}, \mathbf{b}^*_{1,j}\}_{j=1}^{n_1}; s\right)\in C(I), \quad s\in I. \label{def_fin_e1_dis}
\end{equation*}

Suppose that the neural networks  $\mathbf{g}_l:I\to \mathbb{R}^{m_{l,n_l-1}}, \mathbf{f}_l:I\to \mathbb{R}^2,\tilde{e}_l^*:I\to \mathbb{C}$ of grade $l< L$ have been learned and we will learn grade $l+1$.  To this end,  we define the error function of grade $l+1$ by
\begin{equation*}
\tilde{e}_{l+1}\left(\left\{\mathbf{W}_{j}, \mathbf{b}_{ j}\right\}_{j=1}^{n_{l+1}}; s\right):= \tilde{e}^*_{l}(s)-\left[(\mathcal{I}-\lambda\mathcal{K}_{p_{_\kappa}})\mathcal{T}\mathcal{N}_{n_{l+1}}\left(\{\mathbf{W}_{j}, \mathbf{b}_{j}\}_{j=1}^{n_{l+1}};\mathbf{g}_{l}(\cdot)\right)\right](s),\quad s\in I, \label{def_e_q+1_dis}
\end{equation*}
and find  $\{\mathbf{W}_{l+1,j}^*, \mathbf{b}_{l+1,j}^*\}_{j=1}^{n_{l+1}}$ by solving the optimization problem
\begin{equation*} \min \left\{\left\|\tilde{e}_{l+1}\left(\left\{\mathbf{W}_{j}, \mathbf{b}_{j}\right\}_{j=1}^{n_{l+1}}; \cdot\right)\right\|_{N_\kappa}^2: \mathbf{W}_{j}\in \mathbb{R}^{m_{l+1,j}\times m_{l+1,j-1}}, \mathbf{b}_{j}\in \mathbb{R}^{m_{l+1,j}}, j\in \mathbb{N}_{n_{l+1}}\right\}, \label{mult_weight_dis}
\end{equation*}
with $m_{l+1,0}=m_{l,n_l-1}$ and $m_{l+1, n_{l+1}}=2$.
Note that when solving optimization problem \eqref{mult_weight_dis}, the parameters $\mathbf{W}_{\mu,j}^*$, $\mathbf{b}_{\mu,j}^*$, $j=1,2,\dots, n_\mu$, $\mu=1,2,\dots, l$, involved in $\mathbf{g}_{l}$ are fixed.
Then we define the feature of grade $l+1$ by
\begin{equation}
    \mathbf{g}_{l+1}(s):= \mathcal{F}_{n_{l+1}-1}\left(\left\{\mathbf{W}_{l+1,j}^*, \mathbf{b}_{l+1, j}^*\right\}_{j=1}^{n_{l+1}-1};\mathbf{g}_l(s)\right),\quad s\in I, \label{feature_grade_l_1}
\end{equation}
and the solution component of grade $l+1$ by
$$ 
\mathbf{f}_{l+1}(s):=\mathbf{W}_{l+1,n_{l+1}}^*\mathbf{g}_{l+1}(s)+\mathbf{b}_{l+1. n_{l+1}}^*,\quad s\in I.
$$
Substituting equation \eqref{feature_grade_l_1} into the above equation, we see that $\mathbf{f}_{l+1}$ is actually the newly learned neural network stacked on the top of the feature layer  learned in the previous grade. The optimal error of grade $l+1$ is defined by  
\begin{equation*}
\tilde{e}^*_{l+1}(s) := \tilde{e}_{l+1}\left(\left\{\mathbf{W}^*_{l+1,j}, \mathbf{b}^*_{l+1,j}\right\}_{j=1}^{n_{l+1}}; s\right)\in C(I), \quad s\in I. \label{def_fin_e_q+1}
\end{equation*}
We continue this process for  $l<L$. 
The multi-grade DNN approximation for the solution $y$ is given by
$$
\tilde{y}^*_{L}:=\sum_{l=1}^{L} \mathcal{T} \mathbf{f}_l\in C(I).
$$ 
We will show in the next section that the multi-grade DNN solution $\tilde{y}^*_{L}$ is better than the single-grade DNN solution in catching the oscillation features of the exact solution of equation \eqref{fredholm_equation} and thus it has higher approximation accuracy.

\section{Numerical Experiments}
This section is devoted to presentation of numerical experiments that assess and compare the performance of the proposed single-grade learning model and multi-grade learning model with the traditional collocation method. Specifically, we focus on evaluating the methods' performance across varying wavenumbers and sample sizes.  All the  experiments presented in this section were performed on a Ubuntu Server 18.04 LTS 64bit equipped with Intel Xeon Platinum 8255C CPU @ 2.5GHz and NVIDIA Tesla T4 GPU.

In our experiments, we solved the oscillatory Fredholm integral equation \eqref{fredholm_equation_operator} with $\lambda = 0.2$. For comparison purposes, we chose its exact solution as
\begin{equation*}
y(s):=s+(3s^2+2s+1)e^{i\kappa s}+(s+2)e^{-i\kappa s}, \quad s\in I, \label{example}
\end{equation*}
which clearly satisfies the condition \eqref{solution}. The right-hand side $f$ of the integral equation \eqref{fredholm_equation_operator} is then calculated accordingly.
We solved the equation \eqref{fredholm_equation_operator} with the right-hand side specified above using the single-grade, multi-grade models and the traditional collocation method, and compared the accuracy of these methods.
The relative $L_2$ error defined by
\begin{equation*}
\frac{\|y-\tilde{y}\|_2}{\|y\|_2}\approx \frac{1}{\|y\|_2}\left(\frac{2}{l}\left[|y(s_0)-\tilde{y}(s_0)|^2+2\sum_{j=1}^{l-1}|y(s_j)-\tilde{y}(s_j)|^2+|y(s_l)-\tilde{y}(s_l)|^2\right]\right)^{\frac{1}{2}}, \label{relative_error}
\end{equation*}
was used to evaluate the accuracy of these methods, where $y$ and  $\tilde{y}$ denotes the exact solution and an approximate solution, respectively, $l:=20000$, $s_j:=-1+\frac{2j}{l}$, for $j\in \mathbb{Z}_{l+1}$, and $\|y\|_2$ was calculated analytically.


The wavenumber $\kappa$ significantly influences the oscillatory behavior of the solution $y$ and thus impacts the accuracy of its numerical approximations. To evaluate the numerical performance of our proposed model across varying oscillatory levels, we experimented with the $\kappa$ values chosen from  the set $\{100, 150, 200, 250, 300, 350, 400\}$. For the DNN method, we chose $\Gamma=2$, $\gamma=6$, $\beta=1$, $q\in \{1,2\}$ which satisfies the condition \eqref{rule_beta_gamma_q}. Moreover, we set $p_\kappa:=\lceil \gamma \kappa^\beta \rceil=\lceil 6\kappa\rceil$, and investigated the impact of the sample size on the approximation accuracy by considering $N_\kappa:=6\kappa+1$ and  $N_\kappa:=12\kappa+1$,  corresponding to the choices $q:=1$ and $q:=2$, respectively. 

For the training model, we introduce a regularization \cite{Xu:2023aa} to the loss function of optimization problem \eqref{DNN_opt} to address overfitting. Specifically, 
for the single-grade learning model, the training loss is defined as
\begin{equation*}
training\_loss:=\frac{1}{N_\kappa}\sum_{l=1}^{N_\kappa} \left|f(x_l)- ((\mathcal{I}-\mathcal{K}_{p_{_\kappa}}) \mathcal{T}\mathcal{N}_n(\{\mathbf{W}_j, \mathbf{b}_j\}_{j=1}^n);\cdot)(x_l) \right|^2+\mu \sum_{j=1}^n \|\mathbf{W}_j\|_F^2,
\end{equation*}
where $\mu>0$ is the regularization parameter, $\|\cdot\|_F$ is the Frobenius norm and $x_l$ are training data points to be specified later.
A sparse regularization was used in \cite{zeng2022sparse}.
The training loss for the multi-grade learning model is defined in a similar manner.
All models of single-grade and multi-grade  were validated with the validation loss defined as
\begin{equation*}
validation\_loss:=\frac{1}{512}\sum_{l=1}^{512} \left|f(x'_l)- ((\mathcal{I}-\mathcal{K}_{p_{_\kappa}}) Y(x'_l) \right|^2,
\end{equation*}
where $Y:=\tilde{y}^*$ for the single-grade model and  $Y:=\tilde{y}_L^*$ for the multi-grade model, and $x_l'$ are validation data points to be specified.

Now, we specify the training and validation data.

\textbf{Training data:} For each chosen $\kappa$, we  equidistantly chose $N_\kappa$ points $x_j$ from $I$ and compute $f(x_j)$, for $j\in \mathbb{N}_{N_\kappa}$, where $N_\kappa$ is either equal to $6\kappa+1$ or $12\kappa+1$, and thus obtain the training data $\{(x_j,f(x_j))\}_{j=1}^{N_\kappa}\in I\times \mathbb{C}$.


\textbf{Validation data:} The validation set is given by $\{(x'_j,f(x'_j))\}_{j=1}^{512}\in I\times \mathbb{C}$, where $x'_j$, $j\in \mathbb{N}_{512}$, are uniformly distributed on the interval $I$. Note that $N_\kappa$ is not equal to $512$ for any chosen $\kappa$ value.

In our experiments, the activation functions of hidden layers are all chosen to be  the $sin$ function for the single-grade and multi-grade networks.
We chose three single-grade network architectures SGL-1, SGL-2 and SGL-3 as shown in Table \ref{Table:single_grade_structure},
where $[n]$ indicates a fully-connected layer with $n$ neurons. Note that the network SGL-2 is an extension of the network SGL-1 with two additional hidden layers and SGL-3 is an extension of SGL-2 with four additional hidden layers. 

\begin{table}[!ht]
\centering
\begin{tabular}{l|l}
\hline
Methods & Network Structure                                                                     \\ \hline
SGL-1   & $[1]\to[256]\to[256]\to[2]$                             
\\ \hline
SGL-2   & $[1]\to[256]\to[256]\to[128]\to[128]\to[2]$              
\\ \hline
SGL-3   & $[1]\to[256]\to[256]\to[128]\to[128]\to[64]\to[64]\to[32]\to[32]\to[2]$ \\ \hline
\end{tabular}
\caption{Network structure of single-grade learning model}
\label{Table:single_grade_structure}
\end{table}

We use the three grades for the multi-grade learning model   corresponding to SGL-1, SGL-2 and SGL-3, respectively. Their network architectures are described below:

$\mathrm{Grade\,1:} [1]\to[256]\to[256]\to[2]. $

$\mathrm{Grade\,2:} [1]\to[256]_F\to[256]_F\to[128]\to[128]\to[2]. $

$\mathrm{Grade\,3:} [1]\to[256]_F\to[256]_F\to[128]_F\to[128]_F\to[64]\to[64]\to[32]\to[32]\to[2].$\\
Here,  $[n]_F$ indicates a layer having its parameters trained in the previous grades and remained fixed in training of the current grade. 

We now describe the training and the tuning strategies for the single-grade models and the associated multi-grade model. For the single-grade models, we used $3,500$ epochs for training. While for the multi-grade model, we utilized $500$, $1,000$ and $2,000$ epochs for grades 1, 2 and 3, respectively. For all training processes, we uniformly set the initial learning rate $10^{-2}$ and have it exponentially decay to the final learning rate $10^{-7}$. We used regularization parameters  $0, 10^{-6}, 10^{-5}, 10^{-4}$ and batch sizes $64, 128, 256$, and chose the best pair of the hyper-parameters in the sense that with it the model produces the minimum validation error over five independent experiments for each pair of the parameters. We list in Table \ref{hyper_parameter} the best pair of the hyper-parameters found for all models. 

Table \ref{hyper_parameter} shows that MGDL incorporates implicit regularization. For the case $N:=6\kappa+1$, the regularization parameters for MGDL are all 1e-6 very small, and for the case $N:=12\kappa+1$, all regularization parameters for MGDL turn out to be zero, indicating that the MGDL model has a built-in 
regularization feature and additional regularization may not be needed. While the regularization parameters for the deepest model SGL-3 are all 1e-4, hinting that it requires regularization.

\begin{table}[!ht]
\scalebox{0.9}{
\begin{tabular}{|cc|cc|cc|cc|cc|cc|cc|cc|}
\hline
\multicolumn{2}{|c|}{\multirow{2}{*}{}} & \multicolumn{2}{c|}{$\kappa=100$} & \multicolumn{2}{c|}{$\kappa=150$} & \multicolumn{2}{c|}{$\kappa=200$} & \multicolumn{2}{c|}{$\kappa=250$} & \multicolumn{2}{c|}{$\kappa=300$} & \multicolumn{2}{c|}{$\kappa=350$} & \multicolumn{2}{c|}{$\kappa=400$} \\ \cline{3-16} 
\multicolumn{2}{|c|}{} & \multicolumn{1}{c|}{bs} & $\mu$ & \multicolumn{1}{c|}{bs} & $\mu$ & \multicolumn{1}{c|}{bs} & $\mu$ & \multicolumn{1}{c|}{bs} & $\mu$ & \multicolumn{1}{c|}{bs} & $\mu$ & \multicolumn{1}{c|}{bs} & $\mu$ & \multicolumn{1}{c|}{bs} & $\mu$ \\ \hline
\multicolumn{1}{|c|}{\multirow{4}{*}{$N=6\kappa+1$}} & SGL-1 & \multicolumn{1}{c|}{64} & 0 & \multicolumn{1}{c|}{128} & 1e-6 & \multicolumn{1}{c|}{64} & 0 & \multicolumn{1}{c|}{128} & 1e-6 & \multicolumn{1}{c|}{64} & 1e-5 & \multicolumn{1}{c|}{256} & 1e-6 & \multicolumn{1}{c|}{256} & 1e-6 \\ \cline{2-16} 
\multicolumn{1}{|c|}{} & SGL-2 & \multicolumn{1}{c|}{128} & 1e-5 & \multicolumn{1}{c|}{128} & 1e-5 & \multicolumn{1}{c|}{128} & 1e-6 & \multicolumn{1}{c|}{128} & 1e-6 & \multicolumn{1}{c|}{128} & 1e-4 & \multicolumn{1}{c|}{128} & 0 & \multicolumn{1}{c|}{64} & 1e-4 \\ \cline{2-16} 
\multicolumn{1}{|c|}{} & SGL-3 & \multicolumn{1}{c|}{128} & 1e-4 & \multicolumn{1}{c|}{64} & 1e-4 & \multicolumn{1}{c|}{128} & 1e-4 & \multicolumn{1}{c|}{128} & 1e-4 & \multicolumn{1}{c|}{128} & 1e-4 & \multicolumn{1}{c|}{128} & 1e-4 & \multicolumn{1}{c|}{128} & 1e-4 \\ \cline{2-16} 
\multicolumn{1}{|c|}{} & MGDL & \multicolumn{1}{c|}{64} & 1e-6 & \multicolumn{1}{c|}{64} & 1e-6 & \multicolumn{1}{c|}{128} & 1e-6 & \multicolumn{1}{c|}{128} & 1e-6 & \multicolumn{1}{c|}{128} & 1e-6 & \multicolumn{1}{c|}{128} & 1e-6 & \multicolumn{1}{c|}{128} & 1e-6 \\ \hline
\multicolumn{1}{|c|}{\multirow{4}{*}{$N=12\kappa+1$}} & SGL-1 & \multicolumn{1}{c|}{128} & 1e-6 & \multicolumn{1}{c|}{128} & 0 & \multicolumn{1}{c|}{256} & 1e-6 & \multicolumn{1}{c|}{256} & 0 & \multicolumn{1}{c|}{256} & 0 & \multicolumn{1}{c|}{128} & 1e-6 & \multicolumn{1}{c|}{256} & 1e-6 \\ \cline{2-16} 
\multicolumn{1}{|c|}{} & SGL-2 & \multicolumn{1}{c|}{128} & 0 & \multicolumn{1}{c|}{256} & 1e-6 & \multicolumn{1}{c|}{256} & 1e-6 & \multicolumn{1}{c|}{256} & 1e-6 & \multicolumn{1}{c|}{256} & 0 & \multicolumn{1}{c|}{256} & 0 & \multicolumn{1}{c|}{256} & 1e-6 \\ \cline{2-16} 
\multicolumn{1}{|c|}{} & SGL-3 & \multicolumn{1}{c|}{128} & 1e-4 & \multicolumn{1}{c|}{256} & 1e-4 & \multicolumn{1}{c|}{256} & 1e-4 & \multicolumn{1}{c|}{256} & 1e-4 & \multicolumn{1}{c|}{256} & 1e-4 & \multicolumn{1}{c|}{256} & 1e-4 & \multicolumn{1}{c|}{256} & 1e-4 \\ \cline{2-16} 
\multicolumn{1}{|c|}{} & MGDL & \multicolumn{1}{c|}{64} & 0 & \multicolumn{1}{c|}{64} & 0 & \multicolumn{1}{c|}{128} & 0 & \multicolumn{1}{c|}{128} & 0 & \multicolumn{1}{c|}{128} & 0 & \multicolumn{1}{c|}{128} & 0 & \multicolumn{1}{c|}{128} & 0 \\ \hline
\end{tabular}}
\caption{Batch size (bs) and regularization parameter ($\mu$) for SGL-1, SGL-2, SGL-3 and MGDL.}
\label{hyper_parameter}
\end{table}

We compared performance of the DNN models to that of the traditional collocation method using continuous piecewise linear functions and continuous piecewise quadratic functions as bases. For details of the collocation method, the readers are referred to \cite{Atkinson, 2015Multiscale}. We now describe the collocation method using the continuous piecewise polynomial of degree $d$ as the basis for $d=1,2$. For each chosen $\kappa$, let $N_\kappa$ be equal to $6\kappa+1$ or $12\kappa+1$. For each $l\in \mathbb{N}_{N_\kappa}$, the basis function $\phi_l\in C(I)$ is defined to be a polynomial of degree $d$ in the interval $[x_{jd+1},x_{(j+1)d+1}]$ for each $j\in \mathbb{Z}_{(N_\kappa-1)/d}$ and satisfies $\phi_l(x_{j})=\delta_{j,l}$ for any $j\in \mathbb{N}_{N_\kappa}$, where $\delta_{j,l}:=1$, for $j=l$, and $0$ otherwise, for $j,l\in \mathbb{N}_{N_\kappa}$. The collocation method for solving the integral equation 
\eqref{fredholm_equation_operator} with the described bases leads to the algebraic system
$$
\mathbf{G}\mathbf{t}=\mathbf{f},\ 
\mbox{where}
\  \mathbf{G}:=[((\mathcal{I}-\lambda \mathcal{K})\phi_l)(x_j): j,l\in \mathbb{N}_{N_\kappa}], \ 
\mathbf{t}:=[t_j:j\in \mathbb{N}_{N_\kappa}]^\top, \   \mathbf{f}:=[f(x_j): j\in \mathbb{N}_{N_\kappa}]^\top.
$$
By solving the above system, we obtain the coefficients $\mathbf{t}^*:=[t^*_j:j\in \mathbb{N}_{N_\kappa}]^\top$, which gives rise to  the collocation solution  $\hat{y}_d:=\sum_{j=1}^{N_\kappa} t_j^* \phi_j$. In our discussion to follow, we use CM1 and CM2 for the collocation method with piecewise polinomials of degree $d=1$ and $d=2$, respectively.


Relative errors of approximate solutions for all methods are summarized in Table \ref{exp1_compare_result} for wavenumbers $\kappa:=100, 150, 200, 250, 300, 350, 400$ and sample sizes $N_\kappa:=6\kappa+1, 12\kappa+1$. Comparing the three single-grade models SGL-1, SGL-2 and SGL-3, and the multi-grade model MGDL with the two collocation methods, we find that MGDL significantly outperforms all other methods for all wavenumbers and all sample sizes. Among the three single-grade models,  SGL-2 performs the best. Moreover, SGL-2 outperforms CM1 for all wavenumbers and all sample sizes and has slight larger errors than CM2 for most  wavenumbers and sample sizes. However, model SGL-3, which is deeper than SGL-2, performs worse than SGL-2 for all cases, against the expectation that as the depth of the network increases, the expressive power of the neural network should improve. This may be due to the reason that as the neural network becomes deeper, the resulting optimization problem is more difficult to solve, leading to decline in the overall accuracy of the model. Moreover, as it will be shown later in this section, the single-grade deep learning model may suffer from the spectrum bias phenomenon when it is applied to solve an oscillatory integral equation. This serves as motivation for the development of multi-grade learning model.


\begin{table}[!ht]
\centering
\begin{tabular}{|cc|c|c|c|c|c|c|c|}
\hline
\multicolumn{2}{|c|}{} & $\kappa=100$ & $\kappa=150$ & $\kappa=200$ & $\kappa=250$ & $\kappa=300$ & $\kappa=350$ & $\kappa=400$ \\ \hline
\multicolumn{1}{|c|}{\multirow{6}{*}{$N_\kappa=6\kappa+1$}} & CM1 & 1.16e-2 & 1.15e-2 & 1.15e-2 & 1.15e-2 & 1.15e-2 & 1.15e-2 & 1.15e-2 \\ 
\multicolumn{1}{|c|}{} & CM2 & 1.67e-3 & 1.67e-3 & 1.67e-3 & 1.66e-3 & 1.66e-3 & 1.66e-3 & 1.66e-3 \\ \cline{2-9} 
\multicolumn{1}{|c|}{} & SGL-1 & 1.11e-2 & 8.82e-1 & 9.82e-1 & 9.83e-1 & 9.01e-1 & 9.48e-1 & 9.83e-1 \\ 
\multicolumn{1}{|c|}{} & SGL-2 & 1.49e-3 & 1.29e-3 & 1.67e-3 & 1.94e-3 & 3.75e-3 & 3.36e-3 & 5.54e-3 \\ 
\multicolumn{1}{|c|}{} & SGL-3 & 8.11e-3 & 6.91e-3 & 6.80e-3 & 6.61e-3 & 4.92e-3 & 5.92e-3 & 5.38e-3 \\ \cline{2-9} 
\multicolumn{1}{|c|}{} & MGDL & \textbf{4.23e-4} & \textbf{3.09e-4} & \textbf{3.87e-4} & \textbf{3.20e-4} & \textbf{3.05e-4} & \textbf{3.53e-4} & \textbf{3.91e-4} \\ \hline
\multicolumn{1}{|c|}{\multirow{6}{*}{$N_\kappa=12\kappa+1$}} & CM1 & 2.94e-3 & 2.91e-3 & 2.89e-3 & 2.89e-3 & 2.89e-3 & 2.89e-3 & 2.90e-3 \\ 
\multicolumn{1}{|c|}{} & CM2 & 2.10e-4 & 2.10e-4 & 2.10e-4 & 2.09e-4 & 2.09e-4 & 2.09e-4 & 2.08e-4 \\ \cline{2-9} 
\multicolumn{1}{|c|}{} & SGL-1 & 1.22e-1 & 3.17e-1 & 9.81e-1 & 9.82e-1 & 9.82e-1 & 9.83e-1 & 9.84e-1 \\ 
\multicolumn{1}{|c|}{} & SGL-2 & 7.36e-4 & 5.93e-4 & 6.17e-4 & 7.01e-4 & 7.73e-4 & 7.71e-4 & 8.05e-4 \\ 
\multicolumn{1}{|c|}{} & SGL-3 & 3.91e-3 & 3.43e-3 & 3.36e-3 & 3.57e-3 & 3.43e-3 & 3.71e-3 & 4.20e-3 \\ \cline{2-9} 
\multicolumn{1}{|c|}{} & MGDL & \textbf{1.30e-4} & \textbf{8.20e-5} & \textbf{7.81e-5} & \textbf{5.85e-5} & \textbf{4.32e-5} & \textbf{4.54e-5} & \textbf{4.54e-5} \\ \hline
\end{tabular}
\caption{The relative error for CM1, CM2, SGL-1, SGL-2, SGL-2 and MGDL. }
\label{exp1_compare_result}
\end{table}

We further compare the performance of the single-grade learning model and multi-grade learning model for the case when $\kappa=350$ and $N_\kappa=12\kappa+1$. Specifically, we tabulated in Table \ref{error_of_model}  the training loss, the validation loss and the relative error for models SGL-1, SGL-2, SGL-3 and MGDL. We also plot in Figure \ref{figure:training_error} the training loss and the validation loss against the number of epochs. Furthermore, we display in Figure \ref{figure:absolute_value} the absolute error in the time domain between the exact solution and the approximate solutions generated by SGL-3 and grades of MGDL whose network architecture matches that of SGL-3. It is clear that MGDL generates a more accurate approximate solution than SGL-3.


\begin{figure*}[!ht]
    \centering
    \subfigure[SGL-1.]{
		\begin{minipage}{0.45\textwidth}
			\includegraphics[width=1\textwidth]{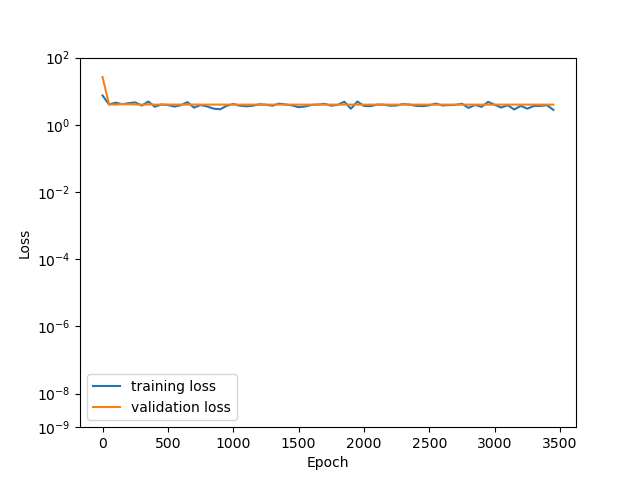}
		\end{minipage}
		\label{subfigure:single_train_error_SGL1}
	}
 \subfigure[SGL-2.]{
		\begin{minipage}{0.45\textwidth}
			\includegraphics[width=1\textwidth]{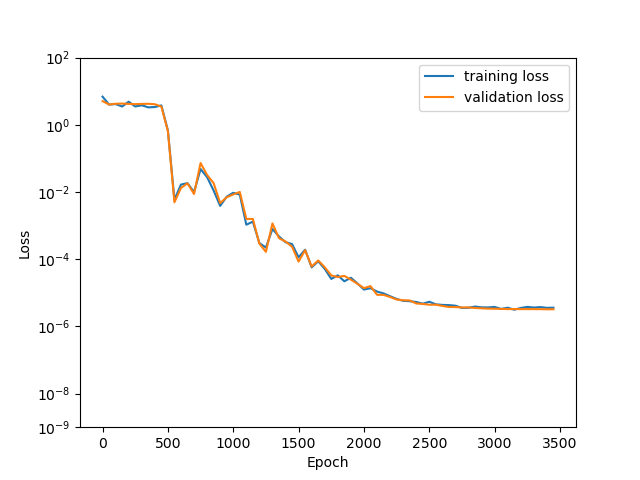}
		\end{minipage}
		\label{subfigure:single_train_error_SGL2}
	}
    \subfigure[SGL-3.]{
		\begin{minipage}{0.45\textwidth}
			\includegraphics[width=1\textwidth]{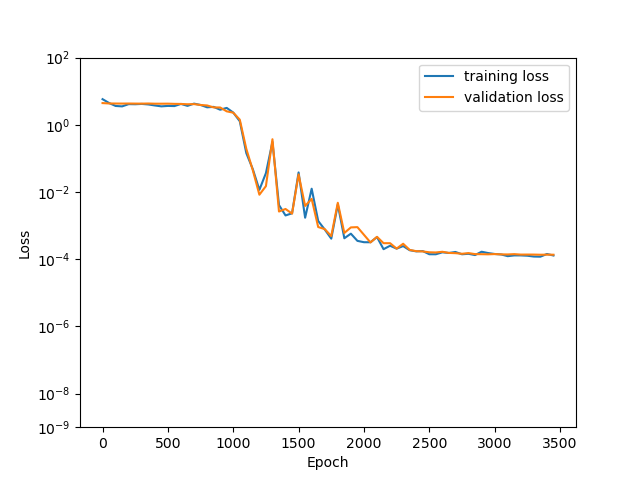}
		\end{minipage}
		\label{subfigure:single_train_error_SGL3}
	}
   \subfigure[MGDL.]{
		\begin{minipage}{0.45\textwidth}
			\includegraphics[width=1\textwidth]{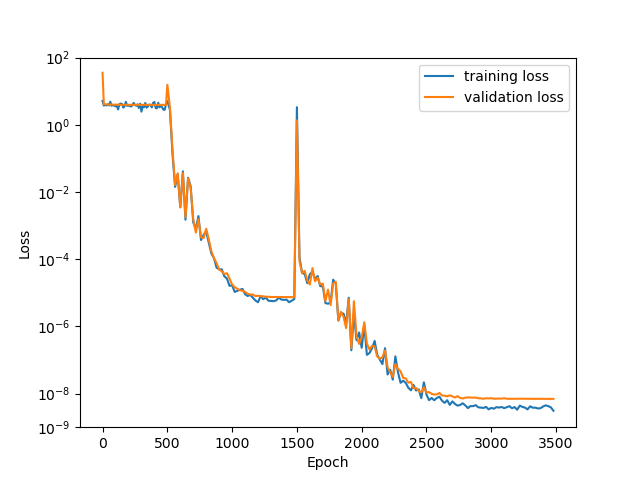}
		\end{minipage}
		\label{subfigure:multi_train_error}
	}

    \caption{The training loss and validation loss for SGL-1, SGL-2, SGL-3 and MGDL for the case $\kappa=350, N_\kappa=12\kappa+1$.}
    \label{figure:training_error}
\end{figure*}

\begin{table}[!ht]
\centering
\begin{tabular}{c||c|c|c||c|c|c}
\hline
 & SGL-1 & SGL-2 & SGL-3 & Grade 1 & Grade 2 & Grade 3 \\ \hline
Training loss & 3.83 & 2.51e-6 & 7.19e-5 & 3.61 & 6.31e-6 & 3.34e-9 \\ 
Validation loss & 3.98 & 2.89e-6 & 6.83e-5 & 3.93 & 8.87e-6 & 6.88e-9 \\ 
Relative error & 9.83e-1 & 7.71e-4 & 3.71e-3 & 9.81e-1 & 1.19e-3 & 4.54e-5 \\ \hline
\end{tabular}
\caption{Training loss, validation loss and relative error of the solution for SGL-1, SGL-2, SGL-3 and MGDL for the case $\kappa=350, N_\kappa=12\kappa+1$.}
\label{error_of_model}
\end{table}

\begin{figure*}[!ht]
    \centering
    \subfigure[The real part for SGL-3.]{
    	\begin{minipage}{0.45\textwidth}
           \includegraphics[width=1\textwidth]{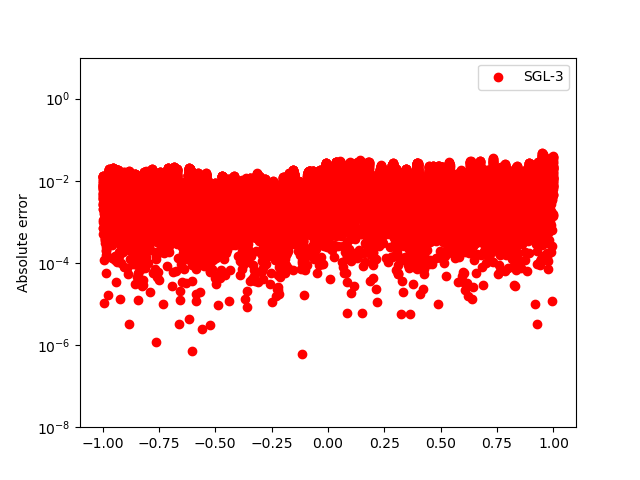}
           \label{real_single}
    	\end{minipage}
      }
      \hfil
    \subfigure[The real part for  MGDL.]{
    	\begin{minipage}{0.45\textwidth}
           \includegraphics[width=1\textwidth]
           {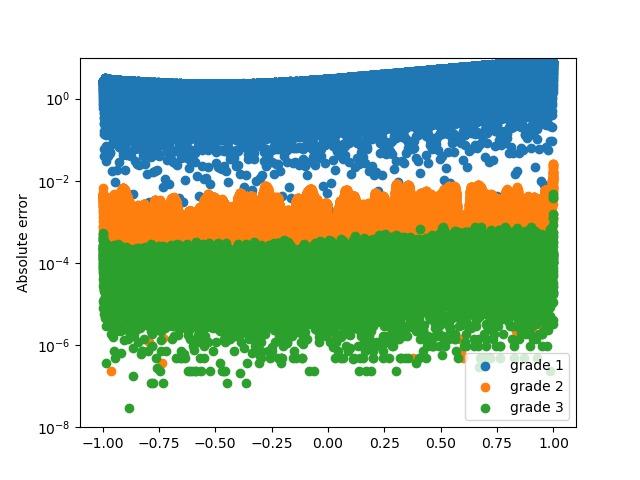}
           \label{real_multi}
           
    	\end{minipage}
    }
   \qquad
   \subfigure[The  imaginary part for SGL-3.]{
    	\begin{minipage}{0.45\textwidth}
           \includegraphics[width=1\textwidth]
           {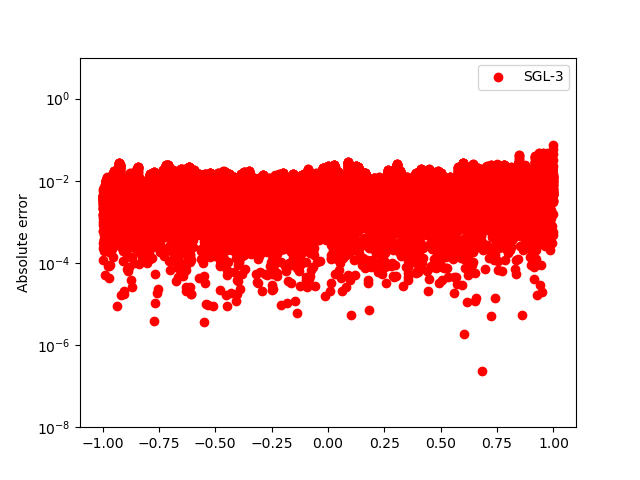}
           \label{imag_single}
    	\end{minipage}
      }
      \hfil
    \subfigure[The imaginary part for MGDL.]{
    	\begin{minipage}{0.45\textwidth}
           \includegraphics[width=1\textwidth]
           {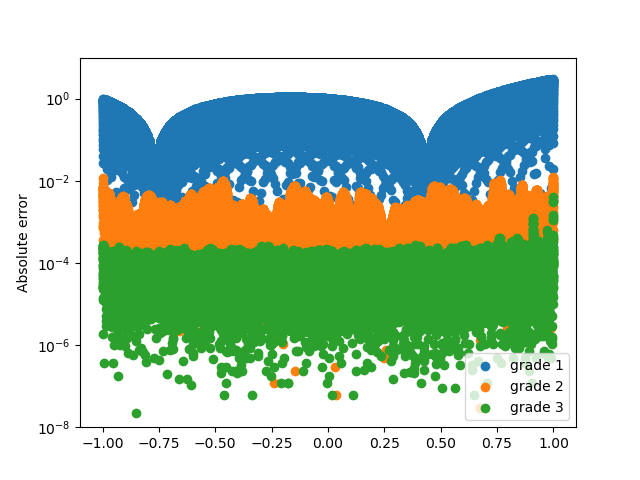}
           \label{imag_multi}
    	\end{minipage}
    }
    \caption{Absolute errors of the approximate solutions of SGL-3 and MGDL at $s_j:=-1+j/10000$, $j\in \mathbb{Z}_{20001}$ for the case $\kappa:=350, N_\kappa:=12\kappa+1$.}
    \label{figure:absolute_value}
\end{figure*}

\begin{figure*}[!ht]
   \centering
    \subfigure[Real part of $\mathcal{T}\mathbf{f}_1$.]{
    	\begin{minipage}{0.3\textwidth}
           \includegraphics[width=1\textwidth]{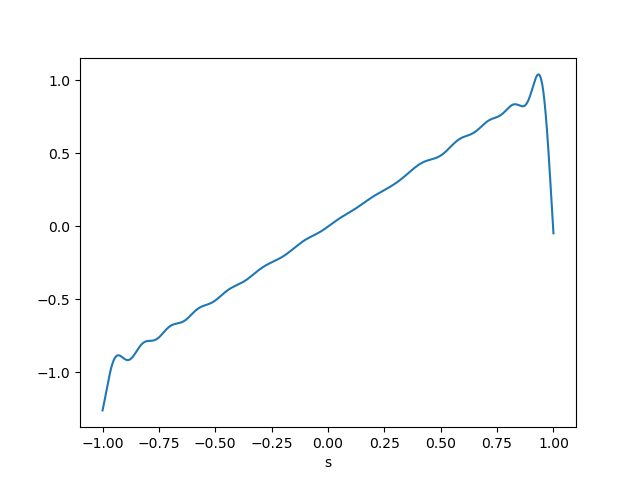}
    	\end{minipage}
     \label{subfigure:grade1_time_real}
      }
      \hfil
   \subfigure[Real part of $\mathcal{T}\mathbf{f}_2$.]{
		\begin{minipage}{0.3\textwidth}
			\includegraphics[width=1\textwidth]{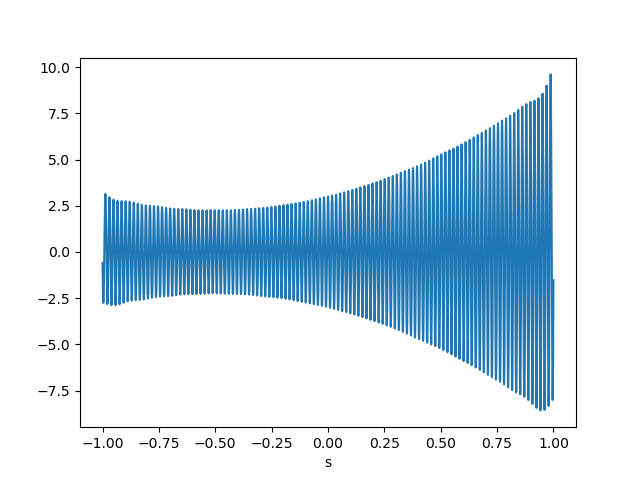}
		\end{minipage}
	}
 \hfil
   \subfigure[Real part of $\mathcal{T}\mathbf{f}_3$.]{
		\begin{minipage}{0.3\textwidth}
			\includegraphics[width=1\textwidth]{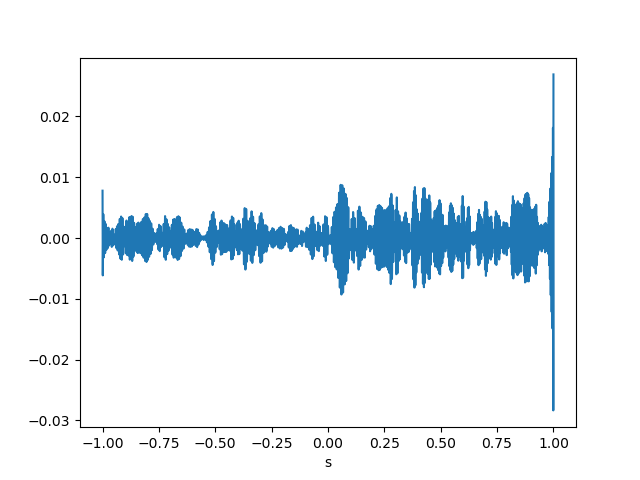}
		\end{minipage}
	}
    \subfigure[Imaginary part of $\mathcal{T}\mathbf{f}_1$.]{
    	\begin{minipage}{0.3\textwidth}
           \includegraphics[width=1\textwidth]{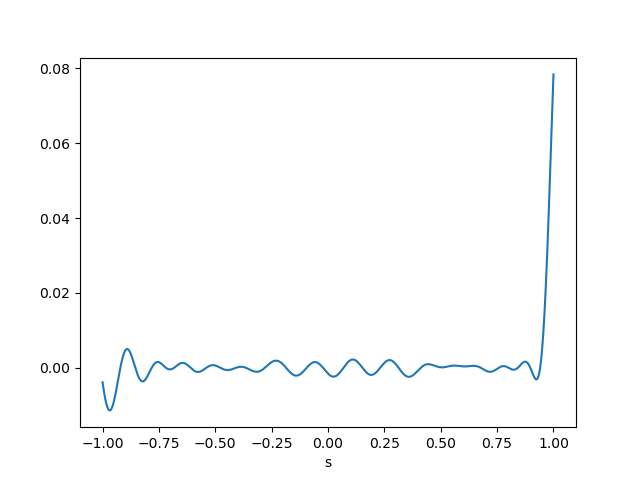}
           \label{subfigure:grade1_time_img}
    	\end{minipage}
      }
      \hfil
   \subfigure[Imaginary part of $\mathcal{T}\mathbf{f}_2$.]{
		\begin{minipage}{0.3\textwidth}
			\includegraphics[width=1\textwidth]{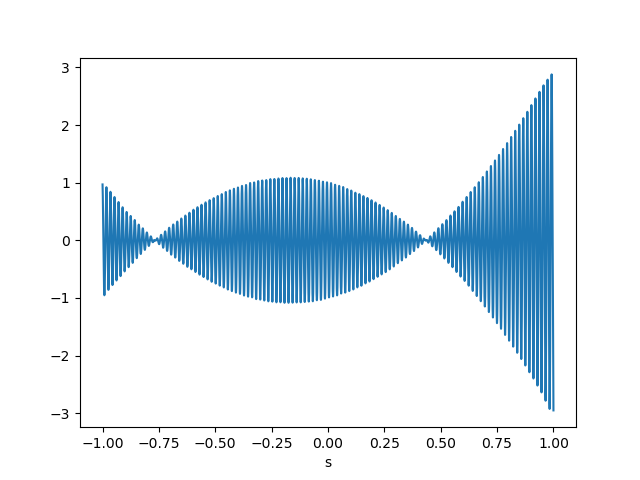}
		\end{minipage}
	}
 \hfil
   \subfigure[Imaginary part of $\mathcal{T}\mathbf{f}_3$.]{
		\begin{minipage}{0.3\textwidth}
			\includegraphics[width=1\textwidth]{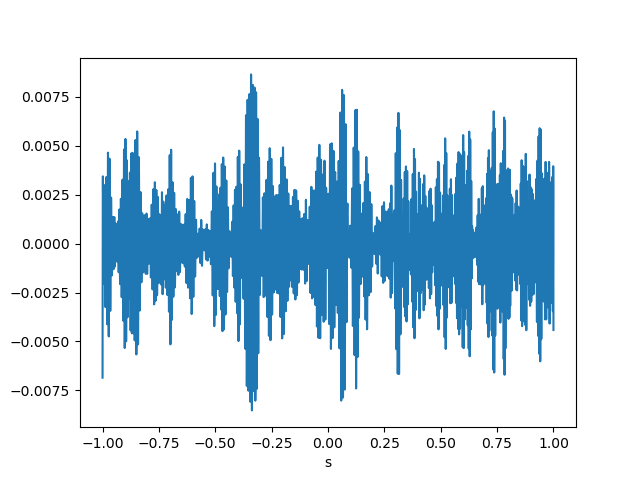}
		\end{minipage}
	}
    \caption{Grade components of MGDL for the case $\kappa=350, N_\kappa=12\kappa+1$.}
    \label{figure:solution_from_model}
\end{figure*}

Figure \ref{figure:solution_from_model} plots different solution components generated by different grades of the MGDL model. It illustrates that the MGDL model can effectively extract the intrinsic multiscale information hidden in the oscillatory solution of the integral equation. This well explains why the MGDL model can overcome the spectral bias from which single-grade learning models suffer.

\begin{figure*}
 \includegraphics[width=1\textwidth]{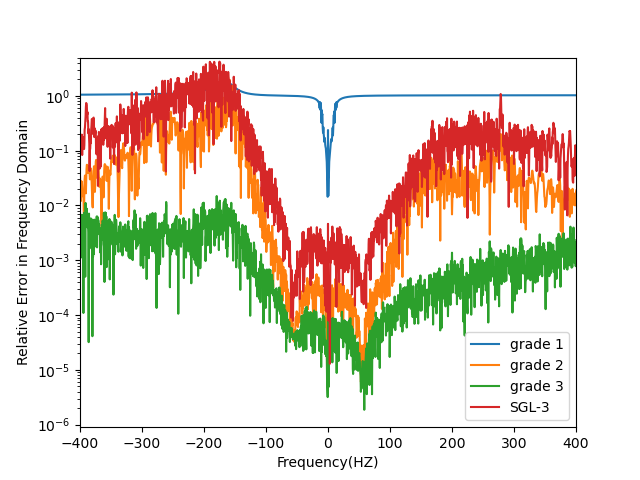}
     \caption{Relative errors of the approximate solutions generated by  SGL-3 and grades 1, 2, 3 of MGDL in the frequency domain
     for the case $\kappa=350, N_\kappa=12\kappa+1$.}
     \label{figure:freq_re}
\end{figure*}

To close this section, we show how the MGDL model improves approximation grade-by-grade looking from the frequency domain. To this end, we denote by $(\mathcal{F}v)(z)$ the fast Fourier transform of a vector $v$, where $z$ denotes the frequency variable. We then use 
\begin{equation*}
    \frac{|\mathcal{F}([y(s_j):j\in \mathbb{N}_{20001}]^T)(z)-\mathcal{F}([Y(s_j):j\in \mathbb{N}_{20001}]^T)(z)|}{|\mathcal{F}([y(s_j):j\in \mathbb{N}_{20001}]^T)(z)|}
\end{equation*}
to compute the relative error of an approximate solution $Y$ at the frequency $z\in \{0.5j-5000.5:j\in \mathbb{N}_{20001}\}$, where $y$ denotes the exact solution and $s_j:=-1+\frac{2j-2}{20000}$, for $j\in \mathbb{N}_{20001}$.
Here, $Y=\tilde{y}^*$ for the single-grade learning model and $Y=Y_l:=\sum_{d=1}^l\mathcal{T}\mathbf{f}_d$ for $l\in \mathbb{N}_L$ for the $l$-th grade solution of the $L$-grade learning model. In Figure \ref{figure:freq_re}, we plot the relative errors between the frequency of the exact solution and that of approximate solutions for SGL-3 and MGDL. Figure \ref{figure:freq_re} demonstrates that the single-grade learning model favours low frequency components.  When the frequency greater than 100, the single-grade learning model gives large errors, which are larger than the errors for the grade 2 solution $Y_2$. On the contrary, errors of the MGDL model reduce grade-by-grade across all frequency levels and for frequencies higher than 100, the errors for the grade 3 solution $Y_3$ reduce remarkably. This pinpoints that the higher grade can capture the high-frequency component of the solution and thus, the MGDL model can overcome the spectral bias from which the single-grade learning suffers.



\section{Conclusive Remarks}
We developed a DNN method for the numerical solution of the oscillatory Fredholm integral equation of the second. The proposed  methodology includes two major components: the numerical quadrature scheme that tailors to computing oscillatory integrals in the context of DNNs and the multi-grade deep learning model that aims at overcoming the spectral bias issue of neural networks. We established the error of the single-grade neural network approximate solution of the equation bounded by the training loss and the quadrature error. We demonstrated by numerical examples that the multi-grade deep learning model is effective in extracting multiscale information of the oscillatory solution and overcoming the spectral bias issue from which the traditional single-grade learning model suffers. 

\bigskip

\noindent{\bf Acknowledgement:} Y. Xu is supported in part by the US National Science Foundation under grant DMS-2208386 and by the US National Institutes of Health under grant R21CA263876. Send all correspondence to Y. Xu.
\bibliographystyle{unsrt}  
\bibliography{REF}

\end{document}